\documentclass[12pt]{amsart}
\usepackage{fullpage,verbatim,amssymb}
\usepackage{hyperref, bbm}
\usepackage{soul}
\usepackage[active]{srcltx}
\usepackage[usenames,dvipsnames]{color}
\hypersetup{colorlinks=true,linkcolor=blue,citecolor=magenta}

\makeatletter
\let\@@pmod\mod
\DeclareRobustCommand{\mod}{\@ifstar\@pmods\@@pmod}
\def\@pmods#1{\mkern4mu({\operator@font mod}\mkern 6mu#1)}
\makeatother

\definecolor{blue}{rgb}{0,0,1}
\definecolor{red}{rgb}{1,0,0}
\definecolor{green}{rgb}{0,.6,.2}
\definecolor{purple}{rgb}{1,0,1}

\long\def\red#1\endred{\textcolor{red}{#1}}
\long\def\blue#1\endblue{\textcolor{blue}{#1}}
\long\def\purple#1\endpurple{\textcolor{purple}{ #1}}
\long\def\green#1\endgreen{\textcolor{green}{#1}}

\newcommand{\ph}{\varphi}
\newcommand{\g}{\gamma}
\renewcommand{\mod}{\textrm{mod}}

\newcommand{\scrL}{\mathcal{L}}
\newcommand{\scrM}{\mathcal{M}}

\newcommand{\scrF}{\mathcal{F}}

\newcommand{\Z}{\mathbb{Z}}
\newcommand{\N}{\mathbb{N}}

\newcommand{\R}{\mathbb{R}}
\newcommand{\C}{\mathbb{C}}

\newcommand{\HH}{\mathbb{H}}
\newcommand{\dd}{( \delta_k }
\newcommand{\ddd}{ \delta_k }

\DeclareMathOperator{\SL}{SL}

\newcommand{\supp}{{\rm Supp}}
\newcommand{\1}{\mathbf{1}}

\newcommand{\sm}{\left(\begin{smallmatrix}}
\newcommand{\esm}{\end{smallmatrix}\right)}
\newcommand{\bpm}{\begin{pmatrix}}
\newcommand{\ebpm}{\end{pmatrix}}

\newtheorem{theorem}{Theorem}
\newtheorem{lemma}[theorem]{Lemma}
\newtheorem{proposition}[theorem]{Proposition}
\newtheorem{corollary}[theorem]{Corollary}
\newtheorem{definition}[theorem]{Definition}

\theoremstyle{remark}

\newtheorem{remark}[theorem]{Remark}

\numberwithin{theorem}{section}
\numberwithin{equation}{section}

\title{$L$-series of harmonic Maass forms and a summation formula for harmonic lifts}
\author{Nikolaos Diamantis} 
\address{University of Nottingham}
\email{nikolaos.diamantis@nottingham.ac.uk}
\author{Min Lee}
\address{University of Bristol}
\email{min.lee@bristol.ac.uk}
\author{Wissam Raji}
\address{American University of Beirut}
\email{wr07@aub.edu.lb}
\author{Larry Rolen}
\address{Vanderbilt University}
\email{larry.rolen@vanderbilt.edu}

\begin{document}

\maketitle

\begin{abstract}
We introduce an $L$-series associated with harmonic Maass forms and prove their functional equations. We establish converse theorems for these $L$-series and, as an application, we formulate and prove a summation formula for the holomorphic part of a harmonic lift of a given cusp form. 
\end{abstract}

\section{Introduction}

The theory of harmonic Maass forms has been a centre of attention in recent years, having led to various striking results. To mention just one example, the following 
harmonic Maass form with Nebentypus of weight $1/2$ for $\Gamma_0(144)$ was a key to the proof of the Andrews-Dragonette conjecture and deep insight into Dyson's ranks \cite{BO1, BO2}:
\begin{equation}\label{AD}q^{-1}+\sum_{n=0}^{\infty}\frac{q^{24n^2-1}}{((1+q^{24})(1+q^{48})\dots (1+q^{24n}))^2}+\int_{-24 \bar z}^{i \infty}\frac{\theta(\tau)d\tau}{\sqrt{-i(\tau+24 z)}}.\end{equation}
Here $q:=e^{2 \pi i z}$ and $\theta(\tau)$ is a certain weight $3/2$ theta series.

However, in contrast to the classical theory, where the deeper study of holomorphic modular and Maass forms is often driven by the study of their $L$-series, Dirichlet series have not yet featured prominently in the case of harmonic Maass forms. 
An $L$-series has been associated to special classes of harmonic Maass forms, namely the weakly holomorphic forms, and interesting results about them have been proved \cite{BFK}, but this $L$-series has not been studied as intensely as the modular object themselves. 
Also, to our knowledge, the definition has not been extended to all harmonic Maass forms, that is, for any harmonic Maass forms which are non-holomorphic. 
In particular, with the exception of a result in that direction we will discuss in the next section, a converse theorem for $L$-series of general harmonic Maass forms does not seem to have been formulated and proved.  

In this paper, inspired by the ideas in \cite{Boo}, we address this state of affairs by proposing a definition of $L$-series of general harmonic Maass forms. 
With this definition, we succeed in establishing a converse theorem. To illustrate the idea more clearly, we will outline it in the special case of weakly holomorphic modular forms on $\SL_2(\Z)$.  

First, we let $\scrL$ be the Laplace transform mapping each smooth function $\varphi\colon \R_+ \to \C$ to
\begin{equation}\label{e:Laplace_trans}
(\scrL \varphi)(s)=\int_0^{\infty} e^{-s t} \varphi(t)dt    
\end{equation}
for each $s \in \C$ for which the integral converges absolutely. 

Let $f$ be a weakly holomorphic cusp form of even weight $k$ for $\SL_2(\Z)$ (see \S\ref{prel} for a definition) with expansion
\begin{equation}\label{FourEx1} 
f(z) = \sum_{\substack{n=-n_0 \\  n \ne 0}}^\infty a(n) e^{2\pi inz}.
\end{equation} 
Let $\scrF_f$ be the space of test functions $\ph\colon \R_+ \to \C$ 
such that
\begin{equation}
\sum_{\substack{n=-n_0 \\  n \ne 0}}^\infty |a(n)| 
(\mathcal L |\ph|)(2 \pi n)
\end{equation}  
converges. 
Because of the growth of $a(n)$ (see \eqref{coeffbound} below),
the space $\scrF_f$, contains the compactly supported smooth functions on $\R_+$.
Then we define the $L$-series map $L_f\colon \scrF_f \to \C$ by
\begin{equation}\label{e:Lseriesmap_def}
L_f(\ph)=\sum_{\substack{n=-n_0 \\  n \ne 0}}^\infty a(n) (\scrL\ph)(2 \pi n).
\end{equation}
The relation of this definition with the $L$-series associated to holomorphic cusp forms and weakly holomorphic modular forms will be discussed in the next section.  

We will now state our converse theorem in the special case of weakly holomorphic cusp forms for $\SL_2(\Z)$. The general statement for all harmonic Maass forms of all levels (Theorem~\ref{thm:CT1}) and its proof will be given in \S\ref{convth}.
\begin{theorem}\label{n=1}
Let $(a(n))_{n \ge -n_0}$ be a sequence of complex numbers such that $a(n)=O(e^{C \sqrt{n}})$ as $n \to \infty$, for some $C>0.$ For each $z \in \HH$, set
\begin{equation}\label{FEwhf}
f(z) = \sum_{\substack{n=-n_0 \\  n \ne 0}}^\infty a(n) e^{2\pi inz}.
\end{equation} 
Suppose that the function $L_f(\ph)$ defined, for each compactly 
supported smooth $\ph: \R_{+} \to \C$ 
, by \eqref{e:Lseriesmap_def}
satisfies 
\begin{equation}\label{FE}
L_f(\ph)=i^kL_f(\check{\ph})
\end{equation}
where $\check{\ph}$ is 
given by
\begin{equation}\label{e:checkph_def}  
\check{\ph}(x):=x^{k-2} \ph(1/x).
\end{equation} 
Then $f$ is a weakly holomorphic cusp form of weight $k\in \Z$ for $\SL_2(\Z)$.
\end{theorem}

As an example of the way the functional equations and the converse theorem we have established can be used, we present an alternative proof of the classical fact that the $(k-1)$-th derivative of a weight $2-k$ weakly holomorphic form is a weight $k$ weakly holomorphic form (Proposition \ref{prop:toy}). 

The main application of our constructions and methods is a summation formula for harmonic lifts via the operator $\xi_{2-k}$. This operator maps a weight $2-k$ harmonic Maass form $f$ to its ``shadow'' weight $k$ holomorphic cusp form 
\begin{equation} 
\xi_{2-k}f:=2iy^{2-k} \overline{\frac{\partial f}{\partial \bar z}}
\end{equation}
where $z=x+iy$. As Bruinier and Funke showed in \cite{BF}, the operator $\xi_{2-k}$ is surjective, and finding a preimage for a given cusp form is a fundamental problem in the theory of harmonic Maass forms with many arithmetic applications (see, e.g., \cite{book}). However, it is not known in general how to  compute explicitly a ``holomorphic part'' (see \eqref{hmfexp}) of a harmonic Maass form $g$ with a known shadow. Our summation formula then provides information about the behaviour of that ``holomorphic part'' upon the action of test functions,
in terms of the given shadow. 
Here we state in the special case of level $1$ and even weight, but in Section \ref{SFsect} we will state it in prove it in general.
\begin{theorem} \label{SF0} Let $f$ be a weight $k \in 2 \N$ holomorphic cusp form with Fourier expansion
 \begin{equation}\label{FEhol}
f(z)=\sum_{n=1}^{\infty}a(n) e^{2 \pi i n z}.
\end{equation}
Suppose that $g$ is a weight $2-k$ harmonic Maass form such that $\xi_{2-k}g=f$ with Fourier expansion
 \begin{equation}\label{hmfexp}
g(z) = \sum_{\substack{n \ge -n_0}} c^+(n) e^{2\pi inz}+ \sum_{\substack{n < 0}} c^-(n) \Gamma(k-1, -4 \pi n y)e^{2 \pi i nz}.
\end{equation}
where $\Gamma(a, z)$ is the incomplete Gamma function. Then, for every smooth, compactly supported $\ph\colon \R_+ \to \R$, we have
\begin{multline}
\sum_{\substack{n \ge -n_0}} c^+(n) \int_0^{\infty} \ph(y) \left ( e^{-2 \pi n y}-(-iy)^{k-2} e^{-2 \pi n/y}\right ) dy\\
=\sum_{l=0}^{k-2}\sum_{n>0}\overline{a(n)} \bigg( \frac{(k-2)!}{l!}(4 \pi n)^{1-k+l}\int_0^{\infty}e^{-2 \pi n y} y^l \ph(y)dy\\
+\frac{2^{l+1}}{(k-1)}(8 \pi n)^{-\frac{k+1}{2}} \int_0^{\infty} e^{-\pi ny}y^{\frac{k}{2}-1}\ph(y)
M_{1-\frac{k}{2}+l, \frac{k-1}{2}}(2 \pi n y)dy\bigg),
\end{multline}
where $M_{\kappa, \mu}(z)$ is the Whittaker hypergeometric function.
\end{theorem}
As usual with summation formulas (see, e.g. \cite{MS} for an overview) the formulation and derivation of our formula is  based on the use of $L$-series, test functions and integral transforms which are the main features of our overall method. 

As  far as we are aware, this is one of the first instances that summation formulas have appeared in the study of harmonic Maass forms and we are currently working on possible applications of our formula. The applications we are aiming for include information about the growth of the individual coefficients $c^+(n)$ and asymptotic formulas for their moments. 

\section*{Acknowledgements} We are thankful to the referees for their careful reading of the manuscript and their insightful comments.  We thank K. Bringmann and J. Lagarias for very useful feedback and suggestions for further work, as well as Ken Ono for his helpful remarks and encouragement. The first author is partially supported by EPSRC grant EP/S032460/1. The second author was supported by Royal Society University Research Fellowship ``Automorphic forms, $L$-functions and trace formulas''. 
The third author is grateful for the support of the Center for Advanced Mathematical Sciences (CAMS) at AUB.
This work was supported by a grant from the Simons Foundation (853830, LR).
The fourth author is also grateful for support from a 2021-2023 Dean's Faculty Fellowship from Vanderbilt University and to the Max Planck Institute for Mathematics in Bonn for its hospitality and financial support.

\section{Context and previous work}\label{intro2} 
We comment on the relation of our $L$-series with the classical $L$-series of holomorphic cusp forms.
For $s\in \C$, let
\begin{equation}\label{testhol}
I_s(x): = (2\pi)^s x^{s-1}\frac{1}{\Gamma\left(s\right)}.
\end{equation}
Then, for $u>0$ and $\Re(s)>0$, 
\begin{equation}
(\scrL I_s)(u)
= \frac{(2\pi)^s}{\Gamma\left(s\right)}
\int_0^\infty e^{-ut}t^{s-1} dt
= \left(\frac{2\pi}{u}\right)^s.
\end{equation}
Here the Laplace transform of $I_s$ continues as an entire function of $s$. 
Let $f$ be a holomorphic cusp form for $\Gamma_0(N)$ of weight $k \in \Z$ with Fourier expansion \eqref{FEhol}.
Since $a(n)=O_{f, \epsilon} (n^{\frac{k-1}{2}+\epsilon})$, for any $s\in\C$ with  $\Re(s)>\frac{k-1}{2}$,
we have $I_s\in \scrF_f$
and
\begin{equation}
L_f(I_s)=\sum_{n=1}^{\infty}\frac{a(n)}{n^s},
\end{equation}
is the usual $L$-series of $f$.

The relation with the $L$-series of a weakly holomorphic cusp form $f$ is more subtle. In this case, $f$ can be expressed in terms of the Fourier expansion \eqref{FEwhf} where $n_0$ is the largest integer such that $a(-n_0)\neq 0$. 
The associated $L$-series is defined in \cite[(1.5)]{BFK}, for any fixed $t_0>0$, by 
\begin{equation}\label{oldLf}
L(s, f):=\sum_{\substack{n \ge -n_0 \\ n \ne 0}}\frac{a(n)\Gamma(s, 2 \pi n t_0)}{(2 \pi n)^s}+i^k
\sum_{\substack{n \ge -n_0 \\ n \ne 0}}\frac{a(n)\Gamma\left(k-s, \frac{2 \pi n}{t_0}\right)}{(2 \pi n)^{k-s}}
\end{equation}
for all $s \in \C$. The value of $L(s,f)$ is independent of $t_0$. 
Here $\Gamma(s, x)$ is the incomplete gamma function 
\begin{equation}
\Gamma(s, x) = \int_x^\infty t^{s-1} e^{-t} dt \quad (\Re(s)>0)
\end{equation}
which continues entirely as a function in $s\in \mathbb{C}$ for $x\neq 0$. 

For a fixed $T>0$, we define the characteristic function
\begin{equation}
\mathbf 1_T(x) := \begin{cases}
1 & \text{ when } x>T, \\ 
0 & \text{ otherwise. }
\end{cases}
\end{equation}
Then, with $I_s$ defined as in \eqref{testhol}, we have, for $t_0>0$ and $u>0,$ 
\begin{equation}\label{interpr}
\scrL(I_s \mathbf 1_{t_0})(u) 
= \int_0^\infty e^{-ut} I_s(t) \mathbf 1_{t_0}(t) dt = \frac{(2\pi)^s}{\Gamma\left(s\right)} u^{-s}  \int_{ut_0}^\infty e^{-t} t^{s-1} dt
= \frac{\Gamma\left(s, ut_0\right)}{\Gamma\left(s\right)} \left(\frac{2\pi}{u}\right)^s.
\end{equation} 
Although the integral defining $\Gamma\left(s, ut_0\right)$ diverges when $u<0$,
the incomplete gamma function has an analytic continuation giving an entire function of $s$, when $u \neq 0$. 
Therefore, we interpret $\scrL(I_s\mathbf 1_{t_0})(u)$ as the analytic continuation of $\Gamma\left(s, ut_0\right)$. 
By \eqref{asym} and \eqref{coeffbound} below, combined with the asymptotic behaviour of the incomplete gamma function and the Fourier coefficients $a(n)$, 
we deduce that, for any $t_0>0$, 
\begin{equation}\label{wholInt}
L_f(I_s\mathbf 1_{t_0})
= \sum_{\substack{n\geq -n_0\\ n\neq 0}} 
\frac{a(n)}{n^s}\frac{\Gamma\left(s, nt_0\right)}{\Gamma\left(s\right)}. 
\end{equation} 
converges absolutely and gives a non-symmetrised form of the $L$-series \eqref{oldLf}. 
Although the definition of $L$-series of weak Maass forms given in \cite{BFK} (see \eqref{oldLf}) addresses the problem of the exponential growth of the forms and of their Fourier coefficients, the fact that the functional equation of the definition \eqref{oldLf} was ``built into'' its defining formula prevented the meaningful formulation of a converse theorem for such $L$-series. 

The construction we present here makes a converse theorem possible by defining the $L$-series on a broader class of test functions than on $\{I_s \mathbf 1_{t_0}: s \in \C\}$ or, equivalently, the parameter $s\in \C$. 
Furthermore the dependence on the test function goes through the Laplace transform, the essential use of which becomes clearer in the applications (Proposition~\ref{prop:toy}, Theorem~\ref{SF}). 
Our approach should be compared to that of Miyazaki et al. \cite{MSSU} in our respective uses of test functions and of integral transforms (Fourier, in their work, and Laplace in ours). The results we establish here complement theirs, because the
latter deal with standard Maass forms whereas we cover functions of exponential growth and harmonic Maass forms. Our approach seems to be also
related to Miller and Schmid's philosophy of automorphic distributions
(see e.g. \cite{MS}) and we intend to investigate the connection more
precisely in future work.

Recently \cite{ShS}, a converse theorem for harmonic Maass forms was announced, but again its focus was on the special case of harmonic Maass forms of polynomial growth, which, in particular, does not cover the function \eqref{AD}. 
Our theorem, by addressing the case of exponential growth, accounts for the situation of a typical harmonic Maass form. 
For the same reason, the techniques introduced here should be more broadly applicable to the various modular objects of non-polynomial growth that have increasingly been attracting attention the last several years, including Brown's real-analytic modular forms \cite{Br, DD} and higher depth weak Maass forms. In relation to the latter, we aim to investigate the connection of our $L$-series with the sesquiharmonic Maass forms associated, in \cite{BDR}, to non-critical values of classical $L$-functions. 

In this paper, we concentrate on foundational analytic aspects of our $L$-series, but the theory is amenable to the study of specific invariants, such as their special values. For example, in \cite{DR}, the hypothetical ``central $L$-value" attached to the classical $j$-invariant in \cite{BFI} is interpreted as an actual value of the $L$-series defined here. 

Finally, a remark on the unusual lack of reference to meromorphic continuation both in Theorem \ref{DThalf} and in Theorems \ref{thm:CT1}, \ref{CT2}. 
The reason for this is that the $L$-series in this paper is defined on a broad family of test functions that contains the compactly supported functions $\ph$. 
As a result, both $\ph$ and its ``contragredient'' $\check{\ph}$  \eqref{e:checkph_def} belong to the domain of absolute convergence of the $L$-series $L_f$. 
This cannot happen in the case of standard $L$-series of holomorphic cusp forms because there is no value of $s$ for which both $I_s$ (in \eqref{testhol}) and $\check{I_s}$
belong to the domain of absolute convergence of the $L$-series. 

However, it is possible to define our $L$-series on classes of test functions for which the above property does not hold automatically. 
Then, the problem of meromorphic continuation arises naturally and can lead to many interesting questions and applications. 
Theorem~\ref{merom} indicates what form a statement involving meromorphic continuation can take in our setting. 
For the initial applications we are concerned with here though, the main issues lay in other aspects and thus the problem of continuation is not relevant.   

\section{Harmonic Maass forms}\label{prel}
We recall the definition and basic properties of harmonic Maass forms. 
For $k \in \frac12 \Z$ we let $\Delta_k$ denote the weight $k$ hyperbolic Laplacian on $\HH$ given by
\begin{equation} 
\Delta_k:=-4y^2 \frac{\partial}{\partial z} \frac{\partial}{\partial \bar z}+2iky\frac{\partial}{\partial \bar z}, 
\end{equation}
where $z=x+iy$ with $x,y\in\R$. 

For $k \in \Z$, we consider the action $|_k$ of $\SL_2(\R)$ on smooth functions $f\colon \HH \to \C$ on the complex upper half-plane $\HH$, given by 
\begin{equation}\label{e:slashk_def}
(f|_k\gamma)(z):= (cz+d)^{-k} f(\gamma z), \qquad \text{for $\gamma=\bpm a & b \\  c & d \ebpm \in$ SL$_2(\R)$}.
\end{equation} 
Here $\gamma z = \frac{az+b}{cz+d}$ is the M\"obius transformation.  

Now we define the action $|_k$ for $k\in \frac{1}{2}+\Z$.  
We let $\left ( \frac{c}{d} \right )$ be the Kronecker symbol. 
For an odd integer $d$, we set 
\begin{equation}
\epsilon_d:=\begin{cases} 1 & \text{ if } d\equiv 1\bmod{4}, \\
i & \text{ if } d\equiv 3\bmod{4}, 
\end{cases} 
\end{equation}
so that $\epsilon_d^2 = \left(\frac{-1}{d}\right)$. 
We set the implied logarithm to equal its principal branch so that $-\pi <$arg$(z) \le \pi$. 
We define the action $|_k$ of $\Gamma_0(N)$, for $4|N$, on smooth functions $f\colon \HH \to \C$ as follows: 
\begin{equation}\label{e:slashk_halfint}
(f|_k\gamma)(z):= 
\left ( \frac{c}{d} \right ) \epsilon_d^{2k} (cz+d)^{-k} f(\gamma z) \qquad \text{ for all } \gamma=\bpm * & * \\  c & d \ebpm  \in \Gamma_0(N).
\end{equation}
In the case of half-integral weight, Shimura \cite{Sh} uses the formalism of the full metaplectic group for the definition of the action. 
From that more general framework, in the sequel we will only need the following special cases
(see, e.g. the proof of \cite[Proposition~5.1]{Sh}): 
Let $W_M=\sm 0 & -\sqrt{M}^{-1} \\ \sqrt{M} & 0\esm$ for $M\in \mathbb{N}$. We have
\begin{equation}\label{WNinC}
(f|_k W_M)(z)=(f|_k W^{-1}_M)(z)=f(W_Mz) (-i \sqrt M z)^{-k}. 
\end{equation}
For $a\in \mathbb{R}_+$ and $b\in \mathbb{R}$, we have 
\begin{equation}
\left (f \Big |_k \bpm \frac{1}{a} & b \\  0 & a \ebpm \right )(z)=a^{-k}f\left (\frac{z+ba}{a^2}\right ). 
\end{equation}
Notice the extra $-i$ in the formula \eqref{WNinC} 
in the half-integral weight case.

With this notation we now state the definition for harmonic Maass forms.
\begin{definition}\label{hmf} 
Let $N \in \N$  and suppose that $4|N$ when $k \in \frac{1}{2} +\Z.$ Let $\psi$ be a Dirichlet character modulo $N$. A \emph{harmonic Maass form of weight $k$ and character $\psi$ for $\Gamma_0(N)$} is a smooth function $f\colon \HH \to \C$ 
such that: 
\begin{enumerate}
\item[i).] For all $\gamma=\left ( \begin{smallmatrix} * & * \\ *& d \end{smallmatrix} \right ) \in \Gamma_0(N)$, we have $f|_k \gamma=\psi(d)f$.
\item[ii).] $\Delta_k(f)=0$.
\item[iii).] For each $\gamma=\sm * & * \\ c & d \esm \in \SL_2(\Z)$, there is a polynomial $P(z) \in \C[e^{-2\pi iz}]$, such that
\begin{equation} \label{growth}
f(\g z)(cz+d)^{-k}-P(z)=O(e^{-\epsilon y}), \qquad \text{as $y \to \infty$, for some $\epsilon>0.$}
\end{equation}
\end{enumerate}
We let $H_k(N, \psi)$ be the space of weight $k$ harmonic Maass forms with character $\psi$ for $\Gamma_0(N)$. On replacing \eqref{growth} with $f(\g z)(cz+d)^{-k}=O(e^{\epsilon y})$ 
we obtain a space denoted by $H'_k(N, \psi)$.
\end{definition}

To describe the Fourier expansions of the elements of $H_k(N, \psi)$, we recall the definition and the asymptotic behaviour of the incomplete Gamma function. 
For $r, z\in \C$ with $\Re(r)>0$, we define the incomplete Gamma function as 
\begin{equation}
\Gamma(r,z) := \int_{z}^\infty e^{-t}t^{r}\, \frac{dt}{t}.
\end{equation}
When $z\neq 0$, $\Gamma(r, z)$ is an entire function of $r$ (see \cite[\S 8.2(ii)]{NIST}).
We note the asymptotic relation for $x \in \R$ (see \cite[(8.11.2)]{NIST})
\begin{equation}\label{asym}
\Gamma(s, x) \sim x^{s-1}e^{-x} \qquad \text{as $|x| \to \infty.$ }
\end{equation}

With this notation we can state the following theorem due to Bruinier and Funke \cite[(3.2)]{BF} 
\begin{theorem}[\cite{BF}] Let $k \in \frac12 \Z.$ Each $f \in H_k(N, \psi)$ have the absolutely convergent Fourier expansion
\begin{equation}\label{FourEx}
f(z) = \sum_{\substack{n \ge -n_0}} a(n) e^{2\pi inz}+ \sum_{\substack{n < 0}} b(n) \Gamma(1-k, -4 \pi n y)e^{2 \pi i nz}
\end{equation} 
for some $a(n), b(n) \in \C$ and $n_0 \in \mathbb N.$ 
Analogous expansions hold at the other cusps.
\end{theorem}
A subspace of particular importance is the space $S_k^!(N, \psi)$ of \emph{weakly holomorphic cusp forms with weight $k\in 2\Z$ and character $\psi$ for $\Gamma_0(N)$}. 
It consists of $f \in H_k(N, \psi)$ which are holomorphic and have vanishing constant terms at all cusps. 

We finally note (cf. \cite[Lemma~3.4]{BF}) that 
\begin{equation} \label{coeffbound}
a(n)=O(e^{C \sqrt{n}}), \quad  b(-n)=O(e^{C\sqrt{n}}) \qquad \text{as $n \to \infty$ for some $C>0$}.
\end{equation}

\section{$L$-series associated to harmonic Maass forms}  \label{lseries}

Let $C(\R, \C)$ be the space of piece-wise smooth complex-valued functions on $\R$. 
We recall the notation $\scrL\ph$ for the Laplace transform of the function $\ph$ on $\R_+$ given in \eqref{e:Laplace_trans}, when the integral is absolutely convergent.
For $s\in \C$, we define 
\begin{equation}
\varphi_s(x) := \varphi(x) x^{s-1}. 
\end{equation}
Note that $\varphi_1 = \varphi$. 

Let $M$ be a positive integer and $k\in \frac{1}{2}\mathbb{Z}$. 
For each function $f$ on $\mathbb{H}$ given by the absolutely convergent series
\begin{equation}\label{FourExmg}
f(z) = \sum_{n \ge -n_0} a(n) e^{2\pi in\frac{z}{M}}+\sum_{n<0} b(n) \Gamma \left (1-k, \frac{-4 \pi n y}{M} \right )e^{2 \pi i n\frac{z}{M}}, 
\end{equation}  
let $\scrF_f$ be the space of functions $\varphi\in C(\R, \C)$ such that the integral defining $(\scrL\varphi)(s)$ (resp. $(\scrL\varphi_{2-k})(s)$) converges absolutely for all $s$ with $\Re(s) \ge -2 \pi n_0$ (resp. $\Re(s)>0$), 
and the following series converges:  
\begin{equation}\label{Ff}
\sum_{\substack{n \ge -n_0}} |a(n)| (\scrL|\varphi|)\left (2 \pi \frac{n}{M}\right ) 
+ \sum_{n<0} |b(n)|\left (\frac{4\pi |n|}{M}\right )^{1-k}\int_0^{\infty}\frac{(\scrL |\varphi_{2-k}|)\left (\frac{-2\pi n(2t+1)}{M}\right )}{(1+t)^{k}}dt. 
\end{equation}
This definition expresses the condition required for absolute and uniform convergence to be guaranteed in the setting we will be working. We note that this construction is possible for any real $k$.

\begin{remark}
In the proof of Theorem \ref{DThalf}, we will see that, for the functions $f$ we will be considering, the space $\scrF_f$ contains the compactly supported functions.
\end{remark}
With this notation we state the following definition. 
\begin{definition}\label{def:Lf}
Let $M$ be a positive integer and  $k\in \frac{1}{2}\mathbb{Z}$. 
Let $f$ be a function on $\HH$ given by 
the Fourier expansion \eqref{FourExmg}. 
The $L$-series of $f$ is defined to be the map
$L_f\colon \mathcal F_f  \to \C$ such that, for $\varphi\in \scrF_{f} $, 
\begin{multline}\label{defL}
L_f(\ph)=\sum_{n \ge -n_0} a(n) (\scrL \ph)(2 \pi n/M) 
\\ + \sum_{n<0} b(n)(-4\pi n/M)^{1-k}\int_0^{\infty}\frac{(\scrL \varphi_{2-k})(-2\pi n(2t+1)/M)}{(1+t)^k} dt.
\end{multline}
\end{definition}

\begin{remark}\label{postDef}
As mentioned in \S\ref{intro2}, this definition is related with previously defined and studied $L$-series. 
See page \pageref{oldLf}
for details on the precise relation.  The domain of the map $L_f$ can be extended to a larger class of test
functions $\varphi$ to account more directly for series such as \eqref{wholInt}. 
However, for the purposes of this work, $\mathcal F_f$ is sufficient. 
\end{remark}

To prove the converse theorem in the case of non-holomorphic elements of $H_k(N, \psi)$, we will also need the following re-normalised version of the partial derivative in terms of $x$, where $z=x+iy \in \HH$: 
\begin{equation}\label{e:f'_def}
\dd f)(z) := z\frac{\partial f}{\partial x}(z) + \frac{k}{2}f(z).
\end{equation}
The context of this operator is that, in contrast to holomorphic functions, to ensure vanishing of a general eigenfunction $F$ of the Laplacian, it is not enough to show vanishing on the imaginary axis. In addition, it is required that $\partial F/\partial x \equiv 0$ on the imaginary axis. The operator $\delta_k$ enables us to formulate a condition in the converse theorem that leads to that vanishing.

Recalling the Fourier expansion given in \eqref{FourExmg}, we have 
\begin{multline}\label{e:f'_Fourierexp}
\dd f)(z) = \frac{k}{2}f(z)
+ \sum_{n\geq -n_0} a(n) \left(2\pi in \frac{z}{M}\right) e^{2\pi in \frac{z}{M}}
\\ + \sum_{n<0} b(n)  \left(2\pi i n \frac{z}{M}\right) \Gamma\left(1-k,    \frac{-4\pi ny}{M}\right) e^{2\pi in\frac{z}{M}}. 
\end{multline}
Although the expansion of $\delta_k f$ is not of the form \eqref{FourExmg}, we can still assign a class of functions $\scrF_{\delta_k f}$ and an $L$-series map $L_{\delta_k f}: \scrF_{\delta_k f}\to \C$ to it.  
Specifically we let $\scrF_{\delta_k f}$ consist of $\varphi\in C(\R, \C)$ such that the following series converges: 
\begin{multline}
2\pi \sum_{n\geq -n_0} |a(n) n| (\scrL|\varphi_2|)(2\pi n/M) 
\\ + 2\pi \sum_{n<0} |b(n) n| (-4\pi n/M)^{1-k} \int_0^\infty \frac{(\scrL|\varphi_{3-k}|)(-2\pi n(2t+1)/M)}{(1+t)^k} dt. 
\end{multline}
Then, we let $L_{\delta_k f}$ be such that, for $\varphi\in \scrF_{\delta_k f}$,  
\begin{multline}\label{e:Lf'_def}
L_{\delta_k f}(\varphi): = \frac{k}{2} L_f(\varphi)
- \frac{2\pi}{M} \sum_{n\geq -n_0} a(n) n (\scrL\varphi_2)(2\pi n/M) 
\\ - \frac{2\pi}{M} \sum_{n<0} b(n) n (-4\pi n/M)^{1-k} \int_0^\infty \frac{(\scrL\varphi_{3-k})(-2\pi n(2t+1)/M)}{(1+t)^k} dt. 
\end{multline}
This converges absolutely. 

\begin{lemma}\label{lem:LfLf'_int}
Let $f$ be a function on $\HH$ as a series in \eqref{FourExmg}. 
For $\varphi\in \scrF_f$, the $L$-series $L_f(\varphi)$ can be given by
\begin{equation}\label{e:Lf_int}
L_f(\varphi)=\int_0^\infty f(iy) \varphi(y) dy. 
\end{equation}
Similarly, for $\varphi\in \scrF_{\delta_k f}$,
\begin{equation}\label{e:Lf'_int}
L_{\delta_k f}(\varphi)=\int_0^\infty \dd f)(iy) \varphi(y) dy, 
\end{equation}
where $\delta_k f$ is defined in \eqref{e:f'_def} and $L_{\delta_k f}$ in \eqref{e:Lf'_def}. 
\end{lemma}

\begin{proof}
By Definition~\ref{def:Lf}, for $\varphi\in \scrF_f$, 
\begin{multline}
L_f(\ph)=\sum_{n \ge -n_0} a(n) (\scrL \ph)(2 \pi n/M) 
\\+ \sum_{n<0} b(n)(-4\pi n/M)^{1-k}\int_0^{\infty}\frac{(\scrL \varphi_{2-k})(-2\pi n(2t+1)/M)}{(1+t)^k} dt
\end{multline}
and this series converges absolutely. 
Since $\varphi\in \scrF_f$, we can interchange the order of summation and integration and write the ``holomorphic" part of the series $L_f(\varphi)$, according to 
\begin{equation}
(\scrL\varphi)\left (\frac{2\pi n}{M}\right ) = \int_0^\infty \varphi(y)e^{-2\pi n\frac{y}{M}}dy.
\end{equation}
For the remaining part, 
thanks to
\begin{equation}
\Gamma(a, z)=z^{a}e^{-z}\int_0^{\infty}\frac{e^{-zt}}{(1+t)^{1-a}}dt \qquad \text{(valid for $\Re(z)>0$)}
\end{equation}
(cf. \cite[(8.6.5)]{NIST}) we can interchange the order of integration to re-write the ``non-holomorphic" part of the series $L_f(\varphi)$, according to
\begin{equation}\label{nonhol}
\int_0^{\infty}\Gamma\left(1-k, -4 \pi n \frac{y}{M}\right)e^{-2 \pi n \frac{y}{M}}\ph(y)dy
= \left (\frac{-4\pi n}{M} \right )^{1-k}\int_0^{\infty}\frac{\scrL\varphi_{2-k}
 \left (\frac{-2\pi n(2t+1)}{M}\right )}{(1+t)^k} dt. 
\end{equation}

The same proof works for $L_{\delta_k f}(\varphi)$. 
\end{proof}

Our goal in the remainder of this section is to state and prove the functional equation of the $L$-series $L_f(\varphi)$, when $f\in H_k(N, \psi)$. 

Let $f$ be a function on $\mathbb{H}$ with the given Fourier expansion \eqref{FourExmg} with $M=1$. 
Let $D$ be a positive integer and let $\chi$ be a Dirichlet character modulo $D$.
We define the ``twist" $f_\chi$ by the Dirichlet character $\chi$ which has a similar series expansion  \eqref{e:fchi_Fourier} given below, with $M=D$ in \eqref{FourExmg}, 
and then we have the corresponding $L$-series $L_{f_\chi}(\varphi)$ as in \eqref{defLchi} below.  
Then, under the assumption that $f$ is an element of the space $H_k(N, \psi)$ of weight $k$ harmonic Maass forms for level $N$ and character $\psi$, we state and prove the functional equation of the $L$-series of $f_\chi$. 
Note that $\chi$ is not necessarily primitive. 

For a Dirichlet character $\chi$ modulo $D$, for each $n \in \Z$, we define the generalized Gauss sum
\begin{equation}\label{e:gausssum_def}
\tau_{\chi}(n):=\sum_{u \bmod D} \chi(u)e^{2 \pi i n\frac{u}{D}}.
\end{equation}
Let $f$ be a function on $\HH$ with the Fourier expansion \eqref{FourExmg} with $M=1$: 
\begin{equation}
f(z) = \sum_{n \ge -n_0} a(n) e^{2\pi inz}+\sum_{n<0} b(n) \Gamma(1-k, -4 \pi n y)e^{2 \pi i nz}. 
\end{equation} 
Then we define the twisted functions $f_\chi$ as
\begin{multline}\label{e:fchi_Fourier}
f_\chi(z):= D^{\frac{k}{2}} \sum_{u\bmod{D}} \overline{\chi(u)} \left(f\big|_k \bpm \frac{1}{\sqrt{D}} & \frac{u}{\sqrt{D}} \\ 0 & \sqrt{D}\ebpm \right)(z)
\\ = \sum_{n\geq -n_0} a(n)\tau_{\bar{\chi}}(n) e^{2\pi i n\frac{z}{D}} 
+ \sum_{n<0} b(n) \tau_{\bar{\chi}}(n) \Gamma\left(1-k, -4\pi n\frac{y}{D}\right) e^{2\pi in\frac{z}{D}}. 
\end{multline}
Then the $L$-series for $f_\chi$ and $\ddd f_\chi$ are 
\begin{multline}\label{defLchi}
L_{f_\chi}(\ph)= \sum_{\substack{n \ge -n_0}} \tau_{\bar \chi}(n)  a(n) (\mathcal L \ph)(2 \pi n/D) \\
+ \sum_{n<0} \tau_{\bar \chi}(n) b(n)(-4\pi n/D)^{1-k}\int_0^{\infty}\frac{\scrL(\ph_{2-k})(-2\pi n(2t+1)/D)}{(1+t)^k}dt 
\end{multline}
and 
\begin{multline}\label{defL'chi}
L_{\ddd f_\chi}(\ph)
=\frac{k}{2}L_{f_\chi}(\ph)
-\frac{2 \pi}{D}\sum_{n \ge -n_0} n\tau_{\bar \chi}(n)  a(n) (\scrL\ph_2)(2 \pi n/D) \\
- \frac{2 \pi}{D}\sum_{\substack{n < 0}}^\infty n\tau_{\bar \chi}(n)b(n)(-4\pi n/D)^{1-k}\int_0^{\infty}\frac{ (\scrL\ph_{3-k})(-2\pi n(2t+1)/D)}{(1+t)^k}dt,  
\end{multline}
for $\varphi\in \scrF_{f_\chi} \cap \scrF_{\ddd(f_\chi)}$. 
By Lemma~\ref{lem:LfLf'_int}, we have 
\begin{align}
& L_{f_\chi}(\varphi)= \int_0^\infty f_\chi(iy) \varphi(y) dy, \label{e:Lfchi_int}
\\ & L_{\ddd f_\chi}(\varphi) = \int_0^\infty (\ddd f_\chi)(iy) \varphi(y) dy.
\label{e:Lf'chi_int}
\end{align}

Before stating the functional equation of the $L_{f_\chi}$, we introduce another notation. 
For each $a \in \frac12 \Z, $ $M \in \mathbb N$ and $\ph\colon \R_{+} \to \C$, we define (note the change in sign convention from earlier in this paper for the action of $W_M$ on functions on $\mathbb H$)
\begin{equation}\label{actionR}
(\ph|_{a}W_M)(x):=(Mx)^{-a} \ph\left(\frac{1}{Mx}\right) \qquad \text{for all $x>0$}.
\end{equation}
Here recall that $W_M=\sm 0 & -\sqrt{M}^{-1} \\\sqrt{M} & 0 \esm$. 
Since this action applies to functions on $\R_+$ and the action \eqref{WNinC} to complex functions, the use of the same notation should not cause a confusion but some caution is advised.

We also define a set of ``test functions" we will be using in most of the remaining results. Let $S_c(\R_{+})$ be a set of complex-valued, compactly supported and piecewise smooth functions on $\R_+$ which satisfy the following condition: 
for any $y\in \R_+$, there exists $\varphi\in S_c(\R_+)$ such that $\varphi(y)\neq 0$. 

We can now prove the functional equation of our $L$-function $L_f(\varphi)$ and its twists.  
\begin{theorem}\label{DThalf}
Fix $k\in \frac{1}{2}\Z$. 
Let $N\in \N$ and let $\psi$ be a Dirichlet character modulo $N$.
When $k\in \frac{1}{2}+\Z$, assume that $4|N$. 
Suppose that $f$ is an element of $H_k(N, \psi)$ with expansion \eqref{FourExmg} and that $\chi$ is a 
character modulo $D$ with $(D, N)=1$. 
Consider the maps $L_{f_\chi}, L_{\ddd f_\chi}\colon \scrF_{f_{ \chi}} \cap  \scrF_{\ddd f_\chi}\to \C$ given in \eqref{defLchi} and \eqref{defL'chi}. 
Set 
\begin{equation}\label{e:gfkWN_def}
g:=f|_kW_{N}    
\end{equation} 
and $\scrF_{f, g} := \left\{\varphi\in \scrF_{f}  \cap \scrF_{\ddd f} \;:\; \varphi|_{2-k} W_N \in  \scrF_{g}  \cap  \scrF_{\ddd g}\right\}. $ 
Then $\scrF_{f, g}\neq \{0\}$ and we have the following functional equations. 
For each $\varphi\in \scrF_{f, g}$, if $k\in \Z$, 
\begin{align}
L_{f_\chi} (\ph)& =i^k
\frac{\chi(-N) \psi(D)}{N^{k/2-1}} L_{g_{\bar \chi}}(\ph|_{2-k}W_N), \label{e:FEN0} \\
L_{\ddd f_\chi}(\ph) & =-i^k
\frac{\chi(-N) \psi(D)}{N^{k/2-1}} L_{\ddd g_{\bar \chi}}(\ph|_{2-k}W_N). \label{e:FEN1}
\end{align}
For each $\varphi\in \scrF_{f, g}$, if $k\in \frac{1}{2}+\Z$, 
\begin{align}
L_{f_\chi}(\ph) & = \psi_D(-1)^{k-\frac{1}{2}} \psi_D(N) 
\frac{\chi(-N) \psi(D)}{\epsilon_D N^{-1+k/2}} L_{g_{\bar{\chi}\psi_D}}(\ph|_{2-k}W_N),  \label{FENhalf}
\\ L_{\ddd f_\chi}(\ph) & = -\psi_D(-1)^{k-\frac{1}{2}}\psi_D(N) 
\frac{\chi(-N) \psi(D)}{\epsilon_D N^{-1+k/2}} L_{\ddd g_{\bar{\chi} \psi_D}}(\ph|_{2-k}W_N).  
\label{FENhalf'}
\end{align}
Here $\psi_D(u) = \left(\frac{u}{D}\right)$ is the real Dirichlet character modulo $D$, given by the Kronecker symbol.
\end{theorem}
\begin{proof} We first note that, exactly as in the classical case, we can show that $g \in H_k(N, \bar \psi)$, if $k \in \Z$ and $g \in H_k(N, \bar \psi \left ( \frac{N}{\bullet}\right ))$, if $k \in \frac12+\Z$.
We further observe that $\scrF_{f, g}$ is non-zero because, clearly, $S_c(\R_{+})$ is closed under the action of $W_N$ and each 
$\scrF_{f
}$
and
$\scrF_{\delta_k f
}$
contains $S_c(\R_{+})$.
Indeed, if $\ph \in S_c(\R_{+}),$ with $\supp(\ph) \subset (c_1, c_2)$ ($c_1, c_2>0$), then, for all $x>0$, \begin{equation}
\mathcal L(|\ph|)(x)=\int_{c_1}^{c_2} |\ph(y)|e^{-x y}dy \ll_{c_1, c_2, \ph} e^{-x c_1}
\end{equation}  
and thus, using \eqref{coeffbound} 
, we deduce that the series in \eqref{Ff} are convergent 
. We further note that if $\varphi \in \scrF_{f}$, then 
$\varphi\in \scrF_{f_{\chi}}$, for all $\chi$. This follows from \eqref{defLchi} and the boundedness of $\tau_{\bar{\chi}}(n).$

Now we prove the functional equations for 
$L_{f_\chi}(\varphi)$ and $L_{\delta_k f_\chi}(\varphi)$. 
Since they depend on whether $k\in \Z$ or $k\in \frac{1}{2}+\Z$, we consider the two cases separately. \\

\noindent\emph{Case I: $k \in\Z$.}
As in the classical case, the definition of $g=f|_k W_N$ and
the identity
\begin{equation}\label{matrmult}
W_N\bpm \frac{1}{\sqrt{D}} & \frac{u}{\sqrt{D}}\\ 0 & \sqrt{D}\ebpm W_N^{-1}
=W_N^{-1}\bpm \frac{1}{\sqrt{D}} & \frac{u}{\sqrt{D}}\\ 0 & \sqrt{D}\ebpm W_N
= \bpm D& -v \\ -Nu & \frac{1+Nuv}{D}\ebpm 
\bpm \frac{1}{\sqrt{D}}& \frac{v}{\sqrt{D}} \\ 0 & \sqrt{D}\ebpm, 
\end{equation}
valid for $u, v \in \Z$ with $\gcd(u, D)=1$ and $Nuv\equiv -1\bmod{D}$, imply that
\begin{equation}\label{trlawN}
f_{\chi}|_k W_N
=
\chi(-N) \psi (D) g_{\bar \chi}. 
\end{equation}
By \eqref{e:Lfchi_int}, by changing the variable $y$ to $\frac{1}{Ny}$, 
and then applying the identity \eqref{trlawN}, 
\begin{multline}
L_{f_\chi}(\varphi)
= \int_0^\infty f_\chi\left(i\frac{1}{Ny}\right)
\varphi\left(\frac{1}{Ny}\right) N^{-1} y^{-2}dy
\\ = \frac{\chi(-N) \psi(D)i^k}{N^{\frac{k}{2}-1}} 
\int_0^\infty g_{\bar{\chi}}(iy) (\varphi|_{2-k}W_N)(y) dy
= \frac{\chi(-N) \psi(D)i^k}{N^{\frac{k}{2}-1}} 
L_{g_{\bar{\chi}}}(\varphi|_{2-k}W_N). 
\end{multline}
This gives the first equality of \eqref{e:FEN0}. 

For the second equality \eqref{e:FEN1}, we applying the operator $\delta_k$ to both sides of \eqref{trlawN}
\begin{equation}
\dd (f_\chi|_k W_N))(z)
= \frac{k}{2}(f_{\chi}|_k W_N)(z)
+z\frac{\partial}{\partial x} (f_{\chi}|_k W_N)(z)
= \chi(-N) \psi(D) \dd g_{\bar{\chi}})(z). 
\end{equation}
For the left hand side, we claim that the 
differential operator $\delta_k$ and action of $W_N$ via $|_k$ almost commute with each other:
\begin{equation}
\dd (f_\chi|_k W_N))(z)
= \frac{k}{2} (f_\chi|_k W_N)(z)
+ z \frac{\partial}{\partial x} \bigg((\sqrt{N} z)^{-k} f_\chi\left(-\frac{1}{Nz}\right)\bigg)
= - (\dd f_\chi)|_kW_N)(z). 
\end{equation}
Then we get
\begin{equation}
((\ddd f_\chi)|_{k} W_N)(z) = -\chi(-N)\psi(D) \dd g_{\bar{\chi}})(z). 
\end{equation}
As above, applying \eqref{e:Lf'chi_int} and using the identity above, we get
\begin{multline}
L_{\ddd f_\chi}(\varphi)
= i^k N^{-\frac{k}{2}+1} \int_0^\infty \dd f_{\chi})|_kW_N)(iy)  
(\varphi|_{2-k} W_N)(y) dy
\\ = - \chi(-N)\psi(D) i^k N^{-\frac{k}{2}+1} \int_0^\infty  \dd g_{\bar{\chi}})(iy)  
(\varphi|_{2-k} W_N)(y) dy
\\ = - \chi(-N)\psi(D) i^k N^{-\frac{k}{2}+1}
L_{\ddd g_{\bar{\chi}}}(\varphi|_{2-k}W_N). 
\end{multline}

\noindent \emph{Case II: $k \in \frac12+\Z$.}
Recall that in this case we assume that $4\mid N$. 
We first note that $g=f|_k W_N$ is a modular form of weight $k$ with character $\bar \psi \cdot \left ( \frac{N}{\bullet}\right)$ for $\Gamma_0(N)$. 
Indeed, 
for each $\gamma=\sm a & b\\ c & d \esm \in \Gamma_0(N)$, the identity
\begin{equation}
W_N \gamma 
= \bpm d & -\frac{c}{N} \\ -bN & a\ebpm W_N 
\end{equation}
implies
\begin{equation} 
g(\gamma z) (cz+d)^{-k}
=\psi(a)\epsilon_a^{-2k} \left ( \frac{-bN}{a}\right )(f|_kW_N)(z)=
\overline{\psi(d)}\epsilon_d^{-2k} \left ( \frac{c}{d}\right )\left ( \frac{N}{d}\right ) g(z)
\end{equation}
since $a \equiv d \mod 4,$ $ad \equiv 1 \mod (-bN)$ and 
$-bc \equiv 1 \mod d$. 

Now, according to Shimura's \cite[Proposition~5.1]{Sh}, we have
\begin{equation}\label{trlawNhalf}
f_{\chi}\left(-\frac{1}{Nz}\right)
\left(-i\sqrt{N}z\right)^{-k}
=\psi_D(-1)^{k-\frac{1}{2}} \psi_D(N)
\frac{\chi(-N) \psi(D)}{\epsilon_D}
g_{\bar{\chi}\psi_D}(z).
\end{equation}

With this,
we obtain, similarly to Case I, the functional equation \eqref{FENhalf} and the functional equation \eqref{FENhalf'}.
\end{proof}
  
As pointed out in the introduction, meromorphic continuation does not play a role in Theorem~\ref{DThalf} and in its converse theorem, Theorem~\ref{thm:CT1}. 
However, it is possible, depending on the application one has in mind, to consider a setting for the theorem that makes meromorphic continuation relevant. 
To illustrate this point we describe such a setting and prove a theorem where meromorphic continuation is part of the conclusion.

Specifically, the test functions, for which the series $L_f(\varphi)$ converges absolutely and the integral $\int_0^\infty f(iy) \varphi(y) dy$ converges (absolutely) are different. 
When $f$ is a holomorphic cusp form of weight $k$ then $\varphi(y)=y^{s+\frac{k-1}{2}-1}$ makes the series $L_f(\varphi)$ converge absolutely for $\Re(s)>1$, but the integral $\int_0^\infty f(iy) \varphi(y) dy$ converges and defines 
a meromorphic function for any $s\in \C$, 
which gives analytic continuation for $L_f(\varphi)$ to any $s\in \C$. 
We discuss the analogue of this phenomenon of the $L$-series in the remainder of this section. 

Recall that $\varphi_s(x) = \varphi(x)x^{s-1}$. 
Then, for $y>0$ and $s\in \C$ with $\Re(s)>\frac{1}{2}$, by Cauchy-Schwarz inequality, 
\begin{equation}\label{e:scrLvarphi_varphis_upper}
(\scrL |\varphi_s|)(y)
\leq \left(\scrL(|\varphi|^2)(y)\right)^{\frac{1}{2}} y^{-\Re(s)+\frac{1}{2}} \left(\Gamma(2\Re(s)-1)\right)^{\frac{1}{2}}.
\end{equation}
Now, for a given function $f$ on $\HH$ with the series expansion \eqref{FourExmg} with $M=1$, 
consider $\varphi\in  \scrF_f$. 
In particular,
\begin{equation}\label{e:scrLfvarphi_condition_fors}
\sum_{\substack{n \ge -n_0}} |a(n)|\left((\scrL |\ph|^2)(2 \pi n)\right)^{\frac{1}{2}} 
+ \sum_{n<0} |b(n)||(-4\pi n)|^{1-k}\int_0^{\infty}
\frac{\left((\scrL |\varphi_{2-k}|^2)(-2\pi n(2t+1))\right )^{\frac{1}{2}}}{(1+t)^k} dt
\end{equation}
converges. Then, with \eqref{e:scrLvarphi_varphis_upper}, we have $\varphi_s\in \scrF_f$ for $\Re(s)> \frac{1}{2} $.

\begin{theorem}\label{merom}
Let $k \in \Z$ and $f\in H_k(N, \psi)$. 
Set $g=f|_kW_N$ and let $n_0 \in \mathbb N$ be such that $f(z)$ and $g(z)$ are $O(e^{2 \pi n_0 y})$ as $y=\Im(z) \to \infty$. Suppose that $\varphi\in C(\R, \C)$ is a non-zero function such that, for some $\epsilon>0$, $\varphi(x)$ and $\varphi(x^{-1})$ are $o(e^{-2\pi (n_0+\epsilon)x})$ as $x\to \infty$. 
We further assume that series  \eqref{e:scrLfvarphi_condition_fors} converges.
Then the series
\begin{equation}\label{e:Lsfvarphi}
L(s, f, \varphi) := L_f(\varphi_{s})
\end{equation}
converges absolutely for $\Re(s)> \frac{1}{2}$, has an analytic continuation to all $s\in \C$ and satisfies the functional equation
\begin{equation}\label{e:Lsfphi_fe}
L(s, f, \varphi) = N^{-s-\frac{k}{2}+1} i^{k} L(1-s, g, \varphi|_{1-k}W_N). 
\end{equation}
\end{theorem}
\begin{proof} 
By the assumption on the growth of $\varphi(y)$ we deduce that
$\scrL(|\varphi|^2)(y)$ converges absolutely for $y \ge -2 \pi n_0$. This combined with the assumption on \eqref{e:scrLfvarphi_condition_fors} and the remarks before the statement of the theorem,  imply that $\varphi_s \in \mathcal F_f$ for
$\Re(s)> \frac12.$
Therefore, recalling the integral representation of $L_f(\varphi_s) = L(s, f, \varphi)$ in \eqref{e:Lf_int}, 
separating the integral at $\sqrt{N}^{-1}$, 
and then changing variables, we get 
\begin{equation}
L(s, f, \varphi) = \int_{\sqrt{N}^{-1}}^\infty f(i (Nx)^{-1}) \varphi((Nx)^{-1}) (Nx)^{-s} \frac{dx}{x} 
+ \int_{\sqrt{N}^{-1}}^\infty f(ix)\varphi(x)x^s \frac{dx}{x}.
\end{equation}
Recall that 
\begin{equation} 
f(i (Nx)^{-1})=(f|_k W_N)(ix) (\sqrt{N} ix)^{k}= g(ix) i^k N^{\frac{k}{2}} x^k 
\end{equation}
and 
\begin{equation} 
\varphi((Nx)^{-1}) = (\varphi|_aW_N)(x) (Nx)^a
\end{equation}
for any $a\in \frac{1}{2}\Z$. 
With $a=1-k$, we get, for $\Re(s)> \frac{1}{2}$ 
\begin{equation}
L(s, f, \varphi)
= i^k N^{-\frac{k}{2}+1-s} 
\int_{\sqrt{N}^{-1}}^\infty 
g(iy) (\varphi|_{1-k} W_N)(x) x^{1-s} \frac{dx}{x}
+ \int_{\sqrt{N}^{-1}}^\infty f(ix) \varphi(x) x^{s} \frac{dx}{x}.
\end{equation}
Because of the growth conditions for $\varphi$ at $0$ and $\infty$, the integrals in the RHS are well-defined for all $s \in \mathbb C$ and give a holomorphic function. 

Since $g|_k W_N = f|_k W_N^2 = (-1)^k f$ and 
$((\varphi|_{1-k}W_N)|_{1-k} W_N)(x) = N^{-1+k} \varphi(x)$, we obtain the functional equation \eqref{e:Lsfphi_fe}.  
\end{proof}

\section{The converse theorem}\label{convth}
To state and prove the converse of Theorem \ref{DThalf}, we recall some further notation from previous sections.

For each $a, b\in \R$ such that $a<b$, we denote by $\1_{[a, b]}(x)$ the characteristic function of the closed interval $[a, b]$.
Further, for each $s \in \C$ and $\varphi\colon \R_{+} \to \C$, we have defined $\varphi_s: \R_{+} \to \C$ so that
$\ph_s(x)=x^{s-1} \ph(x)$ or all $x \in \R_{+}.$ Finally, let $S_c(\R_{+})$ be a set of complex-valued, compactly supported and piecewise smooth functions on $\R_+$ which satisfy the following condition: 
for any $y\in \R_+$, there exists $\varphi\in S_c(\R_+)$ such that $\varphi(y)\neq 0$. 

\begin{theorem}\label{thm:CT1}
Let $N$ be a positive integer and $\psi$ be a Dirichlet character modulo $N$. 
For $j\in \{1, 2\}$, let $(a_j(n))_{n\geq -n_0}$ for some integer $n_0$ and $(b_j(n))_{n<0}$ be sequences of complex numbers such that 
$a_j(n), b_j(n) = O(e^{C\sqrt{|n|}})$ as $|n|\to \infty$ for some constant $C>0$. 
We define smooth functions $f_j: \HH\to \C$ given by the following Fourier expansions associated to the given sequences: 
\begin{equation}
f_j(z) = \sum_{n\geq -n_0} a_j(n) e^{2\pi i n z} 
+ \sum_{n<0} b_j(n) \Gamma\left(1-k, -4\pi n y \right) e^{2\pi i n z}.
\end{equation}
For all $D\in \{1, 2, \ldots, N^2-1\}$, $\gcd(D, N)=1$, let $\chi$ be a Dirichlet character modulo $D$.
For any $\varphi\in S_c(\R_+)$, for any $D$ and $\chi$, we assume that,  
\begin{equation}\label{e:assume_Lf1chi_fe_kinZ}
L_{f_{1 \chi}}(\varphi)
= i^k \frac{\chi(-N)\psi(D)}{N^{\frac{k}{2}-1}} 
L_{f_{2 \overline{\chi}}}(\varphi|_{2-k}W_N) 
\end{equation}
and 
\begin{equation}\label{e:assume_Lf1chi'_fe_kinZ}
L_{\ddd (f_{1\chi})}(\varphi)
= -i^k \frac{\chi(-N) \psi(D)}{N^{\frac{k}{2}-1}} L_{\ddd (f_{2 \overline{\chi}})}(\varphi|_{2-k}W_N), 
\end{equation}
if $k\in \Z$, and 
\begin{equation}\label{e:assume_Lf1chi_fe_kinZ1/2}
L_{f_{1\chi}}(\varphi)
= \psi_D(-1)^{k-\frac{1}{2}} \psi_D(N) 
\frac{\chi(-N) \psi(D)}{\epsilon_D N^{\frac{k}{2}-1}} 
L_{f_{2\overline{\chi}\psi_D}}(\varphi|_{2-k}W_N)
\end{equation}
and 
\begin{equation}\label{e:assume_Lf1chi'_fe_kinZ1/2}
L_{\ddd (f_{1 \chi})}(\varphi)
= - \psi_D(-1)^{k-\frac{1}{2}} \psi_D(N) 
\frac{\chi(-N) \psi(D)}{\epsilon_D N^{\frac{k}{2}-1}} 
L_{\ddd (f_{2 \overline{\chi}\psi_D})}(\varphi|_{2-k}W_N)
\end{equation}
if $k\in \frac{1}{2}+\Z$. 
Here $\psi_D(u) = \left(\frac{u}{D}\right)$ is the real quadratic Dirichlet character given by the  Kronecker symbol. 

Then, the function $f_1$ belongs to $H'_k(\Gamma_0(N), \psi)$ and $f_2=f_1|_k W_N$. 
\end{theorem}

\begin{remark}
There is some freedom in the choice of ``test functions" $\ph$ in this theorem. The compactly supported functions we use in this formulation allow for a cleaner statement and suffices for our applications. Other choices may be more appropriate for different goals and then, additional aspects, such as meromorphic continuation (cf. Theorem~\ref{merom}), may become important. 

In a different direction, we can reduce the size of the set of the test functions required in the converse theorem. For instance, 
we may assume that our functional equations hold only for the family of test functions $\varphi_s(x) = x^{s-1} \varphi(x)$ ($s\in \mathbb{C}$) for a single $\varphi\in S_c(\mathbb{R}_+)$. The converse theorem in this setting can be proved in an essentially identical way as below.  
\end{remark}

\begin{proof}
With the bounds for $a_j(n), b_j(-n)$ and the asymptotic behaviour of $\Gamma(s, x)$ given in \eqref{asym}, 
we have that $f_j(z)$ converges absolutely to a smooth function on $\HH$ for $j\in \{1, 2\}$. 
By the form of the Fourier expansion, $f_1$ and $f_2$ satisfy condition (ii) and condition (iii) at $\infty$ of Definition~\ref{hmf}.  
Likewise, for any Dirichlet character $\chi$ modulo $D$, 
recall that, by definition 
\begin{equation}
f_{j \chi}(z)
= \sum_{\substack{n \ge -n_0}} \tau_{\bar \chi}(n) a_j(n) e^{2\pi inz/D}
+ \sum_{\substack{n < 0}}^\infty \tau_{\bar \chi}(n) b_j(n) \Gamma(1-k, -4 \pi n y/D)e^{2 \pi i nz/D}
\end{equation}
and 
\begin{equation}
\ddd (f_{j \chi})(z) = z\frac{\partial}{\partial x} f_{j \chi}(z) + \frac{k}{2} f_{j \chi}(z), 
\end{equation}
for $j\in \{1, 2\}$, 
are absolutely convergent. 

Our first aim is to show that those functions satisfy the relation \eqref{trlawN} (if $k \in \Z$): 
\begin{equation}
(f_{1\chi}|_k W_N)(z)
= \chi(-N) \psi(D) f_{2\overline{\chi}}(z)
\end{equation}
and \eqref{trlawNhalf} (if $k \in \frac12+\Z$): 
\begin{equation}
(f_{1\chi} |_k W_N)(z)
= \psi_D(-1)^{k-\frac{1}{2}} \psi_D(N) 
\frac{\chi(-N) \psi(D)}{\epsilon_D} f_{2\overline{\chi}\psi_D}(z). 
\end{equation}

Note that for any $s\in \mathbb{C}$ and $\varphi\in S_c(\mathbb{R}_+)$, $\varphi_s(y) = y^{s-1}\varphi(y)\in S_c(\mathbb{R}_+)$. 
We first show that  
$\varphi_s$ satisfies \eqref{Ff} for $f_{j\chi}$ and hence
belongs to $\scrF_{f_{1\chi}} \cap \scrF_{f_2\overline{\chi}}$. 
Indeed, since $\varphi\in S_c(\mathbb{R}_+)$, there exist $0< c_1 < c_2$ and $C>0$ such that ${\rm Supp}(\varphi) \subset [c_1, c_2]$ and $|\varphi(y)|\leq C$ for any $y>0$. Then, for $j\in \{1, 2\}$ and $n>0$,  
\begin{multline}
|a_j(n)|(\scrL|\varphi_s|)\left(\frac{2\pi n}{D}\right)
\leq C |a_j(n)| \int_{c_1}^{c_2} y^{\Re(s)} e^{-2\pi \frac{n}{D} y}  \frac{dy}{y}
\\
\leq C |a_j(n)| e^{-2\pi \frac{n}{D} c_1} (c_2-c_1)\max\{c_1^{\Re(s)-1}, c_2^{\Re(s)-1}\}. 
\end{multline}
Thus,
\begin{multline}
\sum_{n\geq -n_0} |\tau_{\bar{\chi}}(n)| |a_j(n)| (\scrL|\varphi_s|)(2\pi n/D)
 \leq 
\sum_{n=-n_0}^{0} |\tau_{\bar{\chi}}(n)| |a_j(n)|(\scrL|\varphi_s|)(2\pi n/D) 
\\+ C (c_2-c_1)\max\{c_1^{\Re(s)-1}, c_2^{\Re(s)-1}\}
\sum_{n=1}^\infty |\tau_{\bar{\chi}}(n)| |a_j(n)|
e^{-2\pi \frac{n}{D} c_1}
<\infty, 
\end{multline}
for any $s\in \C$ and for any Dirichlet character $\chi$ modulo $D$. 
Likewise, for $n<0$, $t>0$: 
\begin{multline*}
(\scrL|\varphi_{s+1-k}|)\left(\frac{-2\pi n (2t+1)}{D}\right)
\ll 
\int_{c_1}^{c_2} e^{ \frac{2 \pi ny (2t+1)}{D}} y^{\Re(s)} \frac{dy}{y^k}
\ll e^{\frac{2 \pi n c_1 (2t+1)}{D}}
\max\{c_1^{\Re(s)-k}, c_2^{\Re(s)-k}\}
\end{multline*}
and therefore
\begin{multline}
\sum_{n<0} |\tau_{\bar{\chi}}(n)| |b_j(n)| \left|\frac{4\pi n}{D}\right|^{1-k} 
\int_0^\infty \frac{(\scrL|\varphi_{s+1-k}|)\left(-\frac{2\pi n (2t+1)}{D}\right)}{(1+t)^k} dt
\\ \ll \max\{c_1^{\Re(s)-k}, c_2^{\Re(s)-k}\} 
\left ( \int_0^\infty e^{\frac{-4 \pi tc_1}{D}} (1+t)^{-k} dt \right ) 
\sum_{n<0} |\tau_{\bar{\chi}}(n)| |b_j(n)| 
\left(\frac{4\pi |n|}{D}\right)^{1-k} e^{\frac{-2\pi |n|c_1}{D} }
\end{multline}
converge for any $s\in \C$ and for any Dirichlet character $\chi$ modulo $D$. 
Thus $\varphi_s\in \scrF_{f_{1\chi}} \cap \scrF_{f_{2\bar{\chi}}}$ 
and by applying Weierstrass theorem, we see that $L_{f_{j\chi}}(\varphi_s)$ is an analytic function on $s\in \C$. 

This allows us to interchange summation and integration as in Lemma~\ref{lem:LfLf'_int}
and, with Mellin inversion, 
\begin{equation}\label{e:fjchi_Mellin} 
f_{j\chi}(iy) \varphi(y) =  \frac{1}{2\pi i}\int_{(\sigma)} L_{f_{j\chi}}(\varphi_s) y^{-s} ds, 
\end{equation}
for all $\sigma\in \R$. 
In the same way, we see that $L_{\ddd (f_{j \chi})}(\varphi_s)$ is an analytic function for $s\in \C$ and deduce 
\begin{equation}
\ddd (f_{j \chi})(iy) \varphi(y) 
= \frac{1}{2\pi i} \int_{(\sigma)} L_{\ddd (f_{j \chi})}(\varphi_s) y^{-s} ds. 
\end{equation}

Now we will show that $L_{f_{j\chi}}(\varphi_s)\to 0$ as $|\Im(s)|\to \infty$, uniformly for $\Re(s)$, in any compact set in $\C$.
Indeed, with an integration by parts, we have
\begin{equation}
L_{f_{1\chi}}(\varphi_s)
= \int_0^\infty f_{1\chi}(iy) \varphi(y) y^s \frac{dy}{y}
= - \frac{1}{s} \int_0^\infty \frac{d}{dy}\big(f_{1\chi}(iy) \varphi(y)\big) y^s dy 
\end{equation}
since $\varphi(y)$ 
vanishes in $(0, \epsilon) \cup (1/\epsilon, \infty)$ for some $\epsilon>0.$ 
Then 
\begin{equation}
\left|L_{f_{1\chi}}(\varphi_s)\right| 
\leq \frac{1}{|s|} \int_0^\infty \left|\frac{d}{dy}\big(f_{1\chi}(iy) \varphi(y)\big)\right| y^{\Re(s)} dy  \to 0, 
\end{equation}
as $|\Im(s)|\to \infty$. 
The corresponding fact for $L_{\ddd (f_{1\chi})}(\varphi_s)$ is verified in the same way. 

We can therefore move the line of integration in \eqref{e:fjchi_Mellin} from $\Re(s)=\sigma$ to $k-\sigma$, and then changing the variable $s$ to $k-s$, and get 
\begin{equation}\label{e:f1iy_MI_fe}
f_{1\chi}(iy) \varphi(y) =  \frac{1}{2\pi i}\int_{(k-\sigma)} L_{f_{1\chi}}(\varphi_s) y^{-s} ds
= \frac{1}{2\pi i}\int_{(\sigma)} L_{f_{1\chi}}(\varphi_{k-s}) y^{-k+s} ds.
\end{equation}
Similarly, we also have 
\begin{equation}\label{e:f1iy'_MI_fe}
\ddd (f_{1 \chi})(iy) \varphi(y) 
= \frac{1}{2\pi i}\int_{(\sigma)} L_{\ddd (f_{1 \chi})}(\varphi_{k-s}) y^{-k+s} ds.
\end{equation}
To proceed we separate two cases: $k\in \Z$ or $k\in \Z+\frac{1}{2}$. 

\noindent{\it Case I: $k\in \Z$} 
Applying \eqref{e:assume_Lf1chi_fe_kinZ} to \eqref{e:f1iy_MI_fe}, we get
\begin{equation}
f_{1\chi}(iy)\varphi(y)
= i^k \frac{\chi(-N)\psi(D)}{N^{\frac{k}{2}-1}} 
\frac{1}{2\pi i}\int_{(\sigma)}  
L_{f_{2 \overline{\chi}}}(\varphi_{k-s}|_{2-k}W_N) y^{-k+s}
ds.
\end{equation}
We have that $\varphi_{k-s}|_{2-k} W_N \in \scrF_{f_{1\chi}} \cap \scrF_{f_{2\bar{\chi}}}$ and, for each $y>0$, 
\begin{equation}\label{e:varphi_k-s_WN} 
(\varphi_{k-s}|_{2-k} W_N)(y)
= (Ny)^{k-2} \varphi_{k-s}\left(\frac{1}{Ny}\right)
= (Ny)^{s-1} \varphi\left(\frac{1}{Ny}\right). 
\end{equation}
So we get
\begin{equation}\label{e:Lf2barchi_Mellininv}
L_{f_{2\bar{\chi}}}(\varphi_{k-s}|_{2-k}W_N)
= \int_0^\infty 
f_{2\bar{\chi}}(iy) (\varphi_{k-s}|_{2-k}W_N)(y) dy
= \int_0^\infty 
f_{2\bar{\chi}}(iy) (Ny)^{s-1} \varphi\left(\frac{1}{Ny}\right) dy.
\end{equation}
Then, by the Mellin inversion, 
\begin{equation}
N^{-1} f_{2\bar{\chi}}\left(-\frac{1}{i Ny}\right) \varphi(y)
= \frac{1}{2\pi i}\int_{(\sigma)} L_{f_{2\bar{\chi}}}(\varphi_{k-s}|_{2-k} W_N) y^{s} ds.
\end{equation}
Therefore, 
\begin{equation}
f_{1\chi}(iy)\varphi(y)
= i^k \frac{\chi(-N)\psi(D)}{N^{\frac{k}{2}}} 
y^{-k} f_{2\bar{\chi}}\left(-\frac{1}{i Ny}\right) \varphi(y), 
\end{equation}
Similarly, applying \eqref{e:assume_Lf1chi'_fe_kinZ} to \eqref{e:f1iy'_MI_fe}, we get
\begin{equation}
\ddd (f_{1 \chi})(iy) \varphi(y) 
= i^{k+2} \frac{\chi(-N) \psi(D)}{N^{\frac{k}{2}}} y^{-k} \ddd (f_{1 \bar{\chi}})\left(-\frac{1}{iNy}\right) \varphi(y). 
\end{equation}

Therefore, for $y\in \R_+$ such that $\varphi(y)\neq 0$, we have 
\begin{equation}\label{e:f1chiiy_1} 
f_{1\chi}(iy)
= i^k \frac{\chi(-N)\psi(D)}{N^{\frac{k}{2}}} 
y^{-k} f_{2\bar{\chi}}\left(-\frac{1}{i Ny}\right)
\end{equation}
and 
\begin{equation}\label{e:f1chi'iy_1} 
\ddd (f_{1 \chi})(iy) 
= i^{k+2} \frac{\chi(-N) \psi(D)}{N^{\frac{k}{2}}} y^{-k} \ddd (f_{2 \bar{\chi}})\left(-\frac{1}{iNy}\right). 
\end{equation}
Because of the choice of the set of functions $S_c(\R_+)$, 
the above relation is true for all $y>0$. 

We now define 
\begin{equation}
F_\chi(z) := f_{1\chi}(z) - \chi(-N) \psi(D) (f_{2\bar{\chi}}|_k W_N^{-1})(z). 
\end{equation}
The equations \eqref{e:f1chiiy_1} and \eqref{e:f1chi'iy_1} imply that $F_{\chi} (iy)=0$ and $\frac{\partial}{\partial x} F_\chi(iy)=0$. 
Now, $F_\chi$ is an eigenfunction of the Laplace operator because $f_{1\chi}$ and $f_{2\bar{\chi}}$
are eigenfunctions of the Laplace operator with the same eigenvalue. 
Recall that $f_{1\chi}$ and $f_{2\bar{\chi}}$ are eigenfunctions of the Laplace operator because they are defined as a Fourier series of 
$e^{2\pi inz}$ and $\Gamma(1-k, -4\pi ny) e^{2\pi inz}$. 
Therefore (cf. e.g. \cite[Lemma~1.9.2]{B}), the vanishing of $F$ and $\frac{\partial}{\partial x} F$ on the imaginary axis implies that $F_\chi\equiv 0$, 
and then 
\begin{equation}\label{e:chi}
f_{1\chi} = \chi(-N) \psi(D) (f_{2\bar{\chi}}|_k W_N^{-1}). 
\end{equation}

By \eqref{matrmult} and the identity $f_{1} = f_2|_k W_N^{-1}$ (deduced on applying \eqref{e:chi} with $D=1$) we get 
\begin{equation}
f_{2\bar{\chi}}|_k W_N^{-1}
= \sum_{\substack{v\bmod{D}\\ \gcd(v, D)=1, -Nuv\equiv 1\bmod{D}}} \chi(v) f_1\bigg|_k\bpm D & -u \\ -Nv & \frac{1+Nuv}{D}\ebpm 
\bpm \frac{1}{\sqrt{D}} & \frac{u}{\sqrt{D}} \\ 0 & \sqrt{D}\ebpm. 
\end{equation}
With the definition of $f_{1\chi}$ and $f_{1} = f_2|_k W_N^{-1}$,
we have 
\begin{multline}\label{pre-orth}
f_{1\chi} 
= \sum_{\substack{u\bmod{D}\\ \gcd(u, D)=1}} 
\overline{\chi(u)} f_1\bigg|_k \bpm \frac{1}{\sqrt{D}}& \frac{u}{\sqrt{D}} \\ 0 & \sqrt{D}\ebpm 
\\ = \psi(D)
\sum_{\substack{v\bmod{D}\\ \gcd(v, D)=1, -Nuv\equiv 1\bmod{D}}} \overline{\chi(u)} f_1\bigg|_k\bpm D & -u \\ -Nv & \frac{1+Nuv}{D}\ebpm 
\bpm \frac{1}{\sqrt{D}} & \frac{u}{\sqrt{D}} \\ 0 & \sqrt{D}\ebpm, 
\end{multline}
for $\chi(-Nv)=\overline{\chi(u)}$. 
By the orthogonality of the multiplicative characters, 
after taking the sum over all characters modulo $D$, we deduce that, for each $u$ and $v$ such that $-Nuv \equiv 1\bmod{D}$, we have 
\begin{equation}
f_1 = \psi(D)f_1\bigg|_k \bpm D & -u \\ -Nv & \frac{1+Nuv}{D}\ebpm, 
\end{equation}
which is equivalent to 
\begin{equation}
f_1\bigg|_k \bpm \frac{1+Nuv}{D} & u\\ Nv & D\ebpm 
= \psi(D)f_1. 
\end{equation}

This implies that $f_1$ is  invariant with the character $\psi$ for the entire $\Gamma_0(N)$ because, by \cite{R}, the following set of matrices generates $\Gamma_0(N)$: 
\begin{equation}\label{Razar}
\bigcup_{m=1}^{N} S_m \cup \{\pm I_2 \}, 
\end{equation}
where, for each positive $m \in \Z$, $S_m$ is the set consisting of one
$\sm t & s \\ Nm & D \esm \in \Gamma_0(N)$ for each $D$ in a set of congruence classes modulo $Nm$. Finally, working as in the classical case, we deduce that $f_1$ is of at most exponential growth at all cusps.

\noindent \emph{Case II: $k \in \frac12+\Z.$} 
Applying \eqref{e:assume_Lf1chi_fe_kinZ1/2} to \eqref{e:fjchi_Mellin}, 
we have 
\begin{multline}
f_{1\chi}(iy) \varphi(y) 
= \frac{1}{2\pi i}\int_{(\sigma)} L_{f_{1\chi}}(\varphi_s) y^{-s} ds 
\\ = \psi_D(-1)^{k-\frac{1}{2}} \psi_D(N) \frac{\chi(-N) \psi(D)}{\epsilon_D N^{\frac{k}{2}-1}}
\frac{1}{2\pi i}\int_{(\sigma)} L_{f_{2\bar{\chi} \psi_D}} (\varphi_{k-s}|_{2-k} W_N) y^{-k+s} ds.  
\end{multline}
By \eqref{e:varphi_k-s_WN} (holding both for $k \in \Z$ and $k \not \in \Z$) and Mellin inversion, 
\begin{equation}
f_{1\chi}(iy) \varphi(y) 
= \psi_D(-1)^{k-\frac{1}{2}} \psi_D(N) \frac{\chi(-N) \psi(D)}{\epsilon_D N^{\frac{k}{2}}}
y^{-k} f_{2\bar{\chi}\psi_D}\left(-\frac{1}{iNy}\right) \varphi(y). 
\end{equation}
This is true for any $\varphi\in S_c(\R_+)$. 
Because of our choice of $S_c(\R_+)$, for any $y>0$, we have 
\begin{align}\label{e:f1chiiy_2}
f_{1\chi}(iy) 
&= \psi_D(-1)^{k-\frac{1}{2}} \psi_D(N) \frac{\chi(-N) \psi(D)}{\epsilon_D N^{\frac{k}{2}}}
y^{-k} f_{2\bar{\chi}}\left(-\frac{1}{iNy}\right)
\\& = \psi_D(-1)^{k-\frac{1}{2}} \psi_D(N) \frac{\chi(-N) \psi(D)}{\epsilon_D} 
(f_{2\bar{\chi}\psi_D}|_kW_N^{-1})(iy). \nonumber
\end{align}
Similarly, by the functional equation for $L_{\ddd (f_{1 \chi})}(\varphi_s)$ given in \eqref{e:assume_Lf1chi'_fe_kinZ1/2}, and applying the above arguments, for any $y>0$, we have 
\begin{equation}\label{e:f1chiiy'_2} 
\ddd (f_{1 \chi})(iy) 
= -\psi_D(-1)^{k-\frac{1}{2}} \psi_D(N) \frac{\chi(-N) \psi(D)}{\epsilon_D N^{\frac{k}{2}}} 
y^{-k} \ddd (f_{2 \bar{\chi}\psi_D})\left(-\frac{1}{iNy}\right).
\end{equation}
We define 
\begin{equation}
F_\chi(z) = f_{1\chi}(z)-\psi_D(-1)^{k-\frac{1}{2}} \psi_D(N) \chi(-N) \psi(D)\epsilon_D^{-1}
(f_{2\bar{\chi}\psi_D}|_{k} W_N^{-1})(z).
\end{equation}
The equations \eqref{e:f1chiiy_2} and \eqref{e:f1chiiy'_2} imply that $F_\chi(iy)=0$ and $\frac{\partial}{\partial x} F_\chi(iy)=0$. 
As in the Case I (for $k\in \Z$), 
since $F_\chi(z)$ is a Laplace eigenfunction, we deduce that $F_\chi(z)=0$, for any Dirichlet character $\chi$ modulo $D$ and we get 
\begin{equation}
f_{1\chi}
= \psi_D((-1)^{k-\frac{1}{2}}N) \chi(-N) \psi(D) \epsilon_D^{-1} f_{2\bar{\chi}\psi_D} |_k W_N^{-1}. 
\end{equation}
With similar arguments as in Case I we get
\begin{multline}\label{lin1/2}
\sum_{\substack{u\bmod{D}\\ \gcd(u, D)=1}} \overline{\chi(u)} f_1\bigg|_k \bpm \frac{1}{\sqrt{D}} & \frac{u}{\sqrt{D}}\\ 0 & \sqrt{D}\ebpm
\\ = \psi(D)  
\sum_{\substack{v\bmod{D}\\ \gcd(v, D)=1\\ -Nuv\equiv 1\bmod{D}}} \overline{\chi(u)} f_1\bigg|_k \bpm D& -u\\ -Nv & \frac{1+Nuv}{D}\ebpm \bpm \frac{1}{\sqrt{D}} & \frac{u}{\sqrt{D}} \\ 0 & \sqrt{D}\ebpm. 
\end{multline}
By the orthogonality of the multiplicative characters, after taking the sum over all characters $\chi$ modulo $D$, we deduce that, for each $u$ and $v$ such that $-Nuv\equiv 1\bmod{D}$, 
\begin{equation} 
f_1 
\\ = \psi(D) f_1\bigg|_k \bpm D& -u\\ -Nv & \frac{1+Nuv}{D}\ebpm. 
\end{equation}
Therefore
\begin{equation} 
f_1\bigg|_k \bpm \frac{1+Nuv}{D} & u\\ Nv & D\ebpm 
\\ = \psi(D)  f_1. 
\end{equation}
The fact that the set \eqref{Razar} generates $\Gamma_0(N)$ implies the theorem in this case too.
\end{proof}

\begin{corollary}\label{CV1cor}
With the notation of Theorem~\ref{thm:CT1}, let $(a_j(n))_{n \ge -n_0}$  ($j=1, 2$) be sequences of complex numbers such that  $a_j(n)=O(e^{C \sqrt{n}})$ as $n \to \infty$, for some $C>0.$ 
Define holomorphic functions $f_j: \HH\to \C$ by the following Fourier expansions: 
\begin{equation}
f_j(z) = \sum_{n\geq -n_0} a_j(n) e^{2\pi i n z}.
\end{equation}
For all $D\in \{1, 2, \ldots, N^2-1\}$, $\gcd(D, N)=1$, let $\chi$ be a Dirichlet character modulo $D$.
For each $D$, $\chi$ and any $\varphi\in S_c(\R_+)$, we assume that, 
\begin{equation}
L_{f_{1 \chi}}(\varphi)
= i^k \frac{\chi(-N)\psi(D)}{N^{\frac{k}{2}-1}} 
L_{f_{2 \overline{\chi}}}(\varphi|_{2-k}W_N) 
\end{equation}
if $k\in \Z$, and 
\begin{equation}
L_{f_{1\chi}}(\varphi)
= \psi_D(-1)^{k-\frac{1}{2}} \psi_D(N) 
\frac{\chi(-N) \psi(D)}{\epsilon_D N^{\frac{k}{2}-1}} 
L_{f_{2\overline{\chi}\psi_D}}(\varphi|_{2-k}W_N)
\end{equation}
if $k\in \frac{1}{2}+\Z$. 
Then, the function $f_1$ is a weakly holomorphic form with weight $k$ and character $\psi$ for $\Gamma_0(N)$, and $f_2=f_1|_k W_N$. 
\end{corollary}

\begin{proof} 
The proof is identical to that of the theorem except that, thanks to the holomorphicity of $f$ and $g$,
\eqref{e:f1chi'iy_1} is not necessary and thus we do not need the functional equations of $L_{\ddd (f_{j \chi})}(\ph)$. 
\end{proof}

In the case of $N=1$ and the trivial character $\psi$ mod $1$, this corollary becomes Theorem~\ref{n=1}. 

\subsection{Alternative converse theorem for integral weight.}
In the case of integer weight, it is possible to formulate the converse theorem so that only primitive characters are required in the statement. However, the number of primitive characters needed would be infinite and the extension to half-integral weights less transparent. We state this theorem and prove it with emphasis on the parts it differs from Theorem~\ref{thm:CT1}. In particular, we will use the recent method of ``three circles" \cite{NO} which extends to real-analytic functions, the classical vanishing result under a condition about the action of infinite order elliptic elements. 

We first introduce the following notation for the
Gauss sum of a character $\chi$ modulo $D$:
\begin{equation}
\tau(\chi):=\sum_{m \mod D} \chi(m)e^{2 \pi i \frac{m}{D}}.
\end{equation}
We recall that, when $\chi$ is primitive, we have $\tau_{\bar \chi}(n)=\chi(n)\tau(\bar \chi)$.

\begin{theorem}\label{CT2}
Let $k \in \Z,$ $N \in \N$ and $\psi$ be a Dirichlet character modulo $N$. 
For $j\in \{1, 2\}$, let $(a_j(n))_{n\geq -n_0}$ for some integer $n_0$ and $(b_j(n))_{n<0}$ be sequences of complex numbers such that 
$a_j(n), b_j(n) = O(e^{C\sqrt{|n|}})$ as $|n|\to \infty$ for some constant $C>0$. 
We define smooth functions $f_j: \HH\to \C$ given by the following Fourier expansions associated to the given sequences: 
\begin{equation}
f_j(z) = \sum_{n\geq -n_0} a_j(n) e^{2\pi i n z} 
+ \sum_{n<0} b_j(n) \Gamma\left(1-k, -4\pi n y \right) e^{2\pi i n z}.
\end{equation}
For all $D \in \N$ ($\gcd(D, N)=1$), all \emph{primitive} Dirichlet characters $\chi$ modulo $D$ and all $\varphi\in S_c(\R_+)$, we assume that, 
\begin{equation}\label{FECT2}
\begin{aligned}
L_{f_{1 \chi}}(\varphi)
&= i^k \frac{\chi(-N)\psi(D)}{N^{\frac{k}{2}-1}} 
L_{f_{2 \overline{\chi}}}(\varphi|_{2-k}W_N) 
\, \,
\text{and} \\
L_{\ddd (f_{1 \chi})}(\varphi)
&= -i^k \frac{\chi(-N) \psi(D)}{N^{\frac{k}{2}-1}} L_{\ddd (f_{2 \overline{\chi}})}(\varphi|_{2-k}W_N).
\end{aligned} 
\end{equation}
Then, the function $f_1$ belongs to $H'_k(\Gamma_0(N), \psi)$ and $f_2=f_1|_k W_N$. 
\end{theorem}
\begin{proof} 
The deduction of \eqref{e:chi} in the proof of Theorem \ref{thm:CT1} does not depend on whether the character $\chi$ is primitive. Therefore, since the other assumptions of the theorems are the same, we deduce
\begin{equation}\label{trlawN2}
(f_{1\chi}|_k W_N)(z)= \chi(-N) \psi(D) f_{2\bar{\chi}}(z).
\end{equation}
Applying this with $Dz$ instead of $z$, we obtain
$$\tilde f_{1\chi} |_k W_{ND^2}=\chi(-N) \psi(D)\frac{\tau(\chi)}{\tau(\bar \chi)} \tilde f_{2 \bar \chi}$$
where, for $j=1, 2,$
$\tilde f_{j \chi}(z):=\frac{\chi(-1)\tau(\chi)}{D} f_{j\chi}(Dz)$.
This coincides with equation \cite[(5.13)]{B} which, by matrix operations, implies that, for each map on $c(r)$ on the non-zero classes modulo $D$ such that $\sum c(r)=0,$ we have
\begin{equation}\label{linrel2}
\sum_{\substack{r \mod D \\  (r, D)=1}} c(r) f_{2} |_{k} 
\bpm D & -r \\ -Nm & t \ebpm 
\bpm 1 & \frac{r}{D} \\ 0 & 1 \ebpm 
=
\sum_{\substack{r \mod D \\  (r, D)=1}} c(r) \psi(D) f_{2}|_{k} 
\bpm 1 & \frac{r}{D} \\ 0 & 1 \ebpm, 
\end{equation}
where, for each $r$, the integers $t$ and $m$ are such that $Dt-Nmr=1$. 
We note that, once we have such an identity for a given choice of the parameters $r$, $t$ and $m$, then it will hold for \emph{any} other $r$, $t$ such that $Dt-Nmr=1$. 
In the proof of \cite[Theorem 1.5.1 ]{B}, equation \eqref{linrel2} implies that,
\begin{equation}
g := f_{2} |_k\gamma-\psi(D) f_{2}  \qquad \text{ where }\gamma=\bpm D & r \\ Nm & t \ebpm. 
\end{equation}
satisfies $g = g |_k M$, for the elliptic element of infinite order
\begin{equation}
M=\bpm 1 & \frac{2r}{D} \\ -\frac{2Nm}{t} & -3+\frac{4}{Dt} \ebpm.
\end{equation} 
Since the argument in \cite{B} relies exclusively on algebraic manipulations in $\C[\Gamma_0(N)]$, it applies in our case as well.
Therefore, $ g_1:= g |_k\gamma^{-1}= f_2 -\psi(D) f_2 |_k\gamma^{-1}$
satisfies
\begin{equation} 
g_1 = g_1 |_{k}\gamma M \gamma^{-1}.
\end{equation}
As mentioned above, this holds for any $r$ and $t$ such that $Dt-Nmr = 1$. 
Let $h_1 := f_2-\psi(D) f_2|_k \tilde{\gamma}^{-1}$ where 
\begin{equation} 
\tilde \gamma=
\bpm D& r+D \\ Nm & t+Nm \ebpm 
=\gamma T. 
\end{equation}
Here $T=\sm 1 & 1\\ 0 & 1\esm$ is the usual translation matrix.
Let 
\begin{equation}
\tilde M = \bpm 1 & \frac{2(r+D)}{D} \\ -\frac{2Nm}{(t+Nm)} & -3+\frac{4}{D(t+Nm)} \ebpm. 
\end{equation}
Then we have 
\begin{equation}
h_1 = h_1|_k \tilde{\gamma} \tilde{M} \tilde{\gamma}^{-1}. 
\end{equation}

Now, since $f_2 $ satisfies $ f_2 = f_2|_kT,$ we have that 
\begin{equation} 
h_1 = f_2-\psi(D) f_2 |_k T^{-1} \gamma^{-1}=g_1 .
\end{equation}
We claim that the elliptic elements of infinite order $\tilde \gamma \tilde M \tilde \gamma^{-1}$ and $\gamma M \gamma^{-1}$ do not have any fixed points in common. Clearly this is equivalent to the claim that $T \tilde M T^{-1}$ and $M$
do not share any fixed points. Indeed, the former has fixed points
\begin{equation} \frac{1}{DmN}\left (1-Dt \pm \sqrt{1-D(2+mNr)(mN+t)+D^2t(mN+t)}\right )
\end{equation}
and the latter:
\begin{equation}
\frac{1}{DmN}\left (1-Dt \pm \sqrt{1-D(2+mNr)t+D^2t^2)}\right ).
\end{equation}
Their discriminants differ by $DNm \ne 0$.

Therefore, the real-analytic function $g_1$ is invariant under two infinite order elliptic elements with distinct fixed points and, by \cite[Theorem~3.11]{NO}, it vanishes. The completion of the proof is identical to that of \cite[Theorem~1.5.1]{B}. 
\end{proof}

\subsection{Example of using the converse theorem}\label{toy}
Using the above two theorems, we can give an alternative proof of the classic fact that, if $k \in \mathbb N$ and $f$ is a weight $2-k$ weakly holomorphic cusp form, then the $(k-1)$-th derivative of $f$ is  weakly holomorphic cusp form of weight $k$. \cite[Lemma 5.3]{book}
Our purpose is to give a ``proof of concept'' of the way our constructions work. 

\begin{proposition}\label{prop:toy}
Let $k \in 2\N,$ and let $f \in S_{2-k}^!$ for $\SL_2(\Z)$ with Fourier expansion \eqref{FourEx1}.
Then the function $f_1$ given by
\begin{equation} 
f_1(z)=\sum_{\substack{n=-n_0 \\  n \ne 0}}^\infty a(n) (2 \pi n)^{k-1}q^n
\end{equation}
is an element of $S_{k}^!$.
\end{proposition}

\begin{proof}
Since $f \in S^!_{2-k},$  $n^{k-1}a(n)=O(e^{C \sqrt{n}})$ as $n \to \infty$ for some $C>0$. For $\varphi\in S_c(\R_+)$, 
\begin{equation} 
L_{f_1}(\ph)=\sum_{\substack{n=-n_0 \\  n \ne 0}}^\infty (2 \pi n)^{k-1} a(n) (\mathcal L\ph)(2 \pi n)
=\sum_{\substack{n=-n_0 \\  n \ne 0}}^\infty a(n) (\mathcal L (\alpha(\ph))(2 \pi n)
=L_{f}(\alpha (\ph))
\end{equation}
where 
\begin{equation}\label{e:alphavarphi_def}
\alpha(\ph)(x):=\scrL^{-1}(u^{k-1}(\scrL \ph)(u))(x).
\end{equation}
Now, we note that \cite[4.1(8)]{ERD} gives
\begin{equation*} 
(\scrL \ph^{(k-1)})(u)
= u^{k-1}(\scrL \ph)(u)-u^{k-2}\ph(0)-u^{k-3}\ph'(0)-\dots
= u^{k-1}(\scrL \ph)(u)
\end{equation*}
since $\ph$ is supported in $(c_1, c_2) \subset \R_{>0}$. 
Then 
\begin{equation}\label{formulaalpha}
\alpha(\varphi)
= \scrL^{-1}(u^{k-1}(\scrL\varphi)(u))
= \scrL^{-1}(\scrL\varphi^{(k-1)})
= \varphi^{(k-1)}
\end{equation}
and hence, $\alpha(\ph) \in \scrF_f$. Therefore, Theorem~\ref{DThalf} applies, to give
(for $f\in S^!_{2-k}$) 
\begin{equation}\label{CT1eq}
L_{f_1}(\varphi)=L_f(\alpha(\varphi))
= i^{2-k} L_f(\alpha(\varphi)|_{k}W_1). 
\end{equation}
 Here, recall that 
$(\alpha(\varphi)|_k W_1)(x) = x^{-k} \alpha(\varphi)(x^{-1})$. 
On the other hand, 
\begin{equation}\label{CT2eq}
L_{f_1}(\varphi|_{2-k}W_1)
= L_{f}(\alpha(\varphi|_{2-k} W_1))
\end{equation}
We claim that 
\begin{equation}\label{ident}
\alpha(\varphi)|_{k} W_1
= -\alpha(\varphi|_{2-k} W_1), 
\end{equation}
which is equivalent to 
\begin{equation}\label{ident1}
-u^{k-1}(\scrL(\varphi|_{2-k} W_1))(u)
= \scrL(\alpha(\varphi)|_{k} W_1)(u).
\end{equation}
Since both sides are holomorphic in $u$, 
it suffices to prove the above identity for $u>0$. 
To this end, we let $p_\ell(x)=x^{\ell}$, for $\ell\in \mathbb{Z}$, $\ell\geq 1$. 
By \cite[4.2(3)]{ERD}, for $u>0$, 
we have $\frac{1}{\ell!}(\mathcal{L}p_\ell)(u) = p_{-\ell-1}(u)=u^{-\ell-1}$.
By \eqref{e:alphavarphi_def}, we have 
\begin{equation}
\varphi(x)
= \mathcal{L}^{-1}\bigl(p_{-k+1}\cdot \mathcal{L}\alpha(\varphi)\bigr)(x)
= \frac{1}{(k-2)!}
\mathcal{L}^{-1}\bigl(\mathcal{L}p_{k-2}\cdot \mathcal{L}\alpha(\varphi)\bigr)(x).
\end{equation}
By applying the convolution theorem, we get
\begin{equation}
\varphi(x) = \frac{1}{(k-2)!} \int_0^x (x-t)^{k-2} \alpha(\varphi)(t) dt.
\end{equation}
Then, by two changes of variables, 
\begin{multline}
\mathcal{L}(\varphi|_{2-k}W_1)(u)
= \int_0^\infty x^{k-2}\varphi(x^{-1}) e^{-ux} dx
= \frac{1}{(k-2)!} \int_0^\infty\int_0^{x^{-1}} (1-tx)^{k-2} \alpha(\varphi)(t) dt e^{-ux} dx 
\\ = \frac{1}{(k-2)!} \int_0^\infty\int_0^{1} x^{-1} (1-t)^{k-2} \alpha(\varphi)(tx^{-1}) dt e^{-ux} dx
\\ = \frac{1}{(k-2)!} \int_0^\infty x^{-1}\alpha(\varphi)(x^{-1})  
\biggl(\int_0^{1} (1-t)^{k-2} e^{-uxt} dt\biggr) dx. \end{multline}
With \cite[8.2.7]{NIST}
and \cite[8.4.7]{NIST}
we deduce
\begin{multline}\label{preform}
\mathcal{L}(\varphi|_{2-k}W_1)(u)
= \int_0^\infty x^{-1}\alpha(\varphi)(x^{-1})  
(-ux)^{-k+1}e^{-ux} \biggl(1-e^{ux} \sum_{j=0}^{k-2} \frac{(-ux)^{j}}{j!}\biggr) dx
\\ = (-u)^{-k+1} \mathcal{L}(\alpha(\varphi)|_kW_1)(u)
- \sum_{j=0}^{k-2} \frac{(-u)^{-k+1+j}}{j!} \int_0^\infty x^{k-2-j} \alpha(\varphi)(x) dx.
\end{multline}
Since $\varphi$ is compactly supported we have $\int_0^\infty \varphi^{(\ell)}(x) dx=0$ for $\ell \geq 1$. 
Since $\alpha(\varphi)=\varphi^{(k-1)}$, 
for each $j\in [0, k-2]$, by applying integration by parts, we get
\begin{equation}
\int_0^\infty \alpha(\varphi)(x) x^j dx = 0.
\end{equation}
Then, since $k$ is even, we have shown \eqref{ident1} and thus, \eqref{ident}.
Combining this with \eqref{CT1eq} and \eqref{CT2eq}, we deduce $L_{f_1}(\ph)=i^kL_{f_1}(\varphi|_{2-k}W_1)$, 
which, by Corollary~\ref{CV1cor}, implies that $f_1$ is a weakly holomorphic form with weight $k$ for SL$_2(\Z).$
\end{proof}

\subsection{A summation formula for harmonic lifts}\label{SFsect}
Let $N$ be a positive integer, $\chi$ a Dirichlet character modulo $N$ and $\bar{\chi}$ the complex conjugate of the character $\chi$.
We restrict to integers $k \ge 2$ and let $S_k(N, \bar \chi)$ denote the space of standard holomorphic cusp forms of weight $k$ for $\Gamma_0(N)$ and the central character $\bar \chi$. 
We recall \cite{BF} that the ``shadow operator'' 
$\xi_{2-k}\colon H_{2-k}(N, \chi) \to S_k(N, \bar \chi)$ is given by
\begin{equation} 
\xi_{2-k} := 2iy^{2-k} \frac{\bar \partial}{\partial \bar z}.
\end{equation}
It is an important fact, first proved in \cite{BF}, that $\xi_{2-k}$ is surjective. 
The main object in the next theorem is the inverse image of a given element of $S_k(N, \bar \chi)$.
\begin{theorem} \label{SF} Let $k \in 2\N$ and let $f \in S_k(N, \bar \chi)$ with Fourier expansion
\begin{equation} 
f(z)=\sum_{n=1}^{\infty}a_f(n) e^{2 \pi i n z}.
\end{equation}
Suppose that $g$ is an element of $H_{2-k}(N, \chi)$ such that $\xi_{2-k}g=f$ with Fourier expansion 
\begin{equation}
g(z) = \sum_{\substack{n \ge -n_0}} c_g^+(n) e^{2\pi inz}+ \sum_{\substack{n < 0}} c_g^-(n) \Gamma(k-1, -4 \pi n y)e^{2 \pi i nz}.
\end{equation} 
Then, for every $\ph$ in the space $C^{\infty}_c(\R, \R)$ of piecewise smooth, compactly supported functions on $\R$ with values in $\R$, we have
\begin{multline}\label{557}
\sum_{\substack{n \ge -n_0}} c_g^+(n) \int_0^{\infty} \ph(y) e^{-2 \pi n y} dy-
N^{\frac{k}{2}-1}\sum_{\substack{n \ge -n_0}} c_{g|_{2-k}W_N}^+(n) \int_0^{\infty} \ph(y) (-iy)^{k-2} e^{-2 \pi n/(Ny)} dy
\\=\sum_{l=0}^{k-2}\sum_{n>0}\overline{a_f(n)} \Big ( \frac{(k-2)!}{l!}(4 \pi n)^{1-k+l}\int_0^{\infty}e^{-2 \pi n y} y^l \ph(y)dy\\
+\frac{2^{l+1}}{(k-1)}(8 \pi n)^{-\frac{k+1}{2}} \int_0^{\infty} e^{-\pi ny}y^{\frac{k}{2}-1}\ph(y)
M_{1-\frac{k}{2}+l, \frac{k-1}{2}}(2 \pi n y)dy\Big )
\end{multline}
where $M_{\kappa, \mu}(z)$ is the Whittaker hypergeometric function. For its properties, see \cite[\S 13.14]{NIST}).
\end{theorem}
\begin{remark}
Directly from the definition of $\xi_{2-k}$ we see that $a_f(n)=-\overline{c_g^-(-n)}(4 \pi n)^{k-1}$, for each $n \in \N$.
\end{remark}
\begin{proof} 
We can also check that $C^{\infty}_c(\R, \R) \subset \scrF_f \cap \scrF_g$. 
With \eqref{nonhol}, we deduce that the $L$-series of $g$ can be written, for each $\ph \in C^{\infty}_c(\R, \R)$ as
\begin{equation}\label{decomp}
L_g(\ph)=L_g^+(\ph)-\sum_{n>0} \overline{a_f(n)} (4 \pi n)^{1-k} \int_0^{\infty} \Gamma(k-1, 4 \pi n y) e^{2 \pi n y} \ph(y) dy
\end{equation}
where $L_g^+$ denotes the part corresponding to the holomorphic part of $g$:
\begin{equation} 
L_g^+(\ph) := \sum_{n \ge -n_0} c_g^+(n) \mathcal L \ph (2 \pi n).
\end{equation}
The second sum in \eqref{decomp} can be written as
\begin{equation} 
\sum_{n>0} \overline{a_f(n)} \mathcal L ( \Phi(\ph))(2 \pi n)=\overline{L_f(\Phi(\ph))}
\end{equation} 
where
\begin{equation}\label{e:Phi_1def}  
\Phi(\ph)
=\mathcal{L}^{-1} \left ( (2u)^{1-k} \int_0^{\infty} \Gamma(k-1, 2u y)e^{uy}\ph(y)dy \right ).
\end{equation} 
Therefore,
\begin{equation}\label{decompsimpl}
L_g(\ph)=L_g^+(\ph)-\overline{L_f(\Phi(\varphi))}=L_g^+(\ph)-\overline{L_{\xi_{2-k}g}(\Phi(\varphi))}.
\end{equation}
It is clear from its derivation, that this identity holds for any 
weight $k$ harmonic Maass form $g$ and, in particular, also for $g|_{2-k}W_N.$

Now, Theorem \ref{thm:CT1} applied to $L_g$ implies that $L_g(\ph)=i^{2-k}N^{k/2}L_{g|_{2-k}W_N}(\ph|_{k}W_N)$. Therefore
\begin{equation}\label{decomp0}L_g^+(\ph)-\overline{L_f(\Phi(\ph))}=i^{2-k}N^{k/2}
\left (L_{g|_{2-k}W_N}^+(\ph|_{k}W_N)-\overline{L_{\xi_{2-k}(g|_{2-k}W_N)}(\Phi(\ph|_{k}W_N))} \right ).
\end{equation}
Similarly, 
with the identity 
\begin{equation} 
\xi_{2-k}(g|_{2-k}W_N)|_k{W_N}=\xi_{2-k}(g)|_k W_N|_kW_N=(-1)^kf
\end{equation}
we deduce that 
\begin{equation} 
L_{\xi_{2-k}(g|_{2-k}W_N)}(\Phi(\ph|_{k}W_N))
= i^{-k}N^{1-k/2}L_f(\Phi(\ph|_{k}W_N)|_{2-k}W_N).
\end{equation}
Therefore, \eqref{decomp0} becomes
\begin{equation}\label{decomp1} L_g^+(\ph)-\overline{L_f(\Phi(\ph))}=
i^{2-k}N^{k/2}L_{g|_{2-k}W_N}^+(\ph|_{k}W_N)+N\overline{L_f(\Phi(\ph|_{k}W_N)|_{2-k}W_N))}.
\end{equation} 
To simplify $L_f(\Phi(\ph|_{k}W_N)|_{2-k}W_N))$, we first note that a change of variables followed by an application of 
\cite[4.1(4)]{ERD} gives
\begin{multline}
\scrL^{-1} \left((2u)^{1-k} \int_0^{\infty}\Gamma(k-1, 2 u y) e^{u y} \ph\left(\frac{1}{Ny}\right)\frac{dy}{(Ny)^k} \right) \left(\frac{1}{Nx}\right)
\\ =
N^{-k}\scrL^{-1} \left ((2u/N)^{1-k} \int_0^{\infty}\Gamma(k-1, 2 (u/N) y)e^{(u/N)y} \ph\left(\frac{1}{y}\right)\frac{dy}{y^k} \right)\left(\frac{1}{x}\right) 
\\ =
N^{1-k}\scrL^{-1} \left ((2u)^{1-k} \int_0^{\infty}\Gamma(k-1, 2 u y)e^{uy} \ph\left(\frac{1}{y}\right)\frac{dy}{y^k} \right)\left(\frac{1}{x}\right).
\end{multline}
Then, with \cite[4.1(25)]{ERD}, we obtain
\begin{multline}
\scrL\left(\Phi(\ph|_{k}W)|_{2-k}W_N)\right)(2 \pi n) \\
=
\scrL \left ((Nx)^{k-2}\scrL^{-1} \left((2u)^{1-k} \int_0^{\infty}\Gamma(k-1, 2 u y) e^{u y} \ph\left(\frac{1}{Ny}\right)\frac{dy}{(Ny)^k} \right) \left(\frac{1}{Nx}\right) \right)(2 \pi n)\\
=\frac{(2\pi n)^{\frac{1-k}{2}}}{N}\int_0^{\infty}u^{\frac{k-1}{2}}J_{k-1}(\sqrt{8\pi n u}) (2u)^{1-k}
\int_0^{\infty}\Gamma(k-1, 2uy)e^{uy}\ph(1/y)y^{-k}dydu\\
=\frac{2^{1-k}(2\pi n)^{\frac{1-k}{2}}}{N}\int_0^{\infty}\ph(y)y^{k-2}\int_0^{\infty} u^{\frac{1-k}{2}}J_{k-1}(\sqrt{8\pi n u}) \Gamma(k-1, 2u/y)e^{u/y}dudy.
\end{multline}
The formula \cite[(8.4.8)]{NIST} for the incomplete Gamma function implies that the last expression
equals
\begin{equation}\label{Mf}
\begin{aligned}
&\frac{(8\pi n)^{\frac{1-k}{2}}}{N}\sum_{l=0}^{k-2}\frac{2^l(k-2)!}{l!}\int_0^{\infty}\ph(y)y^{k-2-l}\int_0^{\infty} u^{\frac{1-k}{2}+l}J_{k-1}(\sqrt{8\pi n u}) e^{-u/y}dudy\\
&=\frac{(8\pi n)^{\frac{1-k}{2}}}{N}\sum_{l=0}^{k-2}\frac{2^{l+1}(k-2)!}{l!}\int_0^{\infty}\ph(y)y^{k-2-l}\int_0^{\infty} u^{2-k+2l}J_{k-1}(\sqrt{8\pi n} u) e^{-u^2/y}dudy \\
&=\frac{(8\pi n)^{\frac{-k}{2}}}{N\sqrt{8 \pi n} (k-1)}\sum_{l=0}^{k-2}2^{l+1}\int_0^{\infty}\ph(y)y^{\frac{k}{2}-1}e^{-\pi n y} M_{1-\frac{k}{2}+l, \frac{k-1}{2}}(2 \pi n y)dy
\end{aligned}
\end{equation}
where, for the last equality we used \cite[6.8(8)]{ERD}.

Finally, with \cite[(8.4.8)]{NIST} again, we deduce
\begin{equation}
\begin{aligned}\label{Gf}
L_f(\Phi(\ph))&=\sum_{n>0} a_f(n) \mathcal L ( \Phi(\ph))(2 \pi n)\\
&=
\sum_{l=0}^{k-2}\sum_{n>0}a_f(n) \frac{(k-2)!}{l!}(4 \pi n)^{1-k+l}\int_0^{\infty}e^{-2 \pi n y} y^l \ph(y)dy.
\end{aligned}
\end{equation}

Replacing \eqref{Mf} and \eqref{Gf} into \eqref{decomp1}, we derive the theorem.
\end{proof}

{\footnotesize
\nocite{*}
\bibliographystyle{amsalpha}
\bibliography{referenceLfunctions}

\providecommand{\bysame}{\leavevmode\hbox to3em{\hrulefill}\thinspace}
\providecommand{\MR}{\relax\ifhmode\unskip\space\fi MR }
\providecommand{\MRhref}[2]{%
  \href{http://www.ams.org/mathscinet-getitem?mr=#1}{#2}
}
\providecommand{\href}[2]{#2}
\begin{thebibliography}{EMOT54}

\bibitem[BDR13]{BDR}
Kathrin Bringmann, Nikolaos Diamantis, and Martin Raum, \emph{Mock period
  functions, sesquiharmonic {M}aass forms, and non-critical values of
  {$L$}-functions}, Adv. Math. \textbf{233} (2013), 115--134.

\bibitem[BF04]{BF}
Jan~Hendrik Bruinier and Jens Funke, \emph{On two geometric theta lifts}, Duke
  Math. J. \textbf{125} (2004), no.~1, 45--90.

\bibitem[BFI15]{BFI}
Jan~H. Bruinier, Jens Funke, and \"{O}zlem Imamo\={g}lu, \emph{Regularized
  theta liftings and periods of modular functions}, J. Reine Angew. Math.
  \textbf{703} (2015), 43--93.

\bibitem[BFK14]{BFK}
Kathrin Bringmann, Karl-Heinz Fricke, and Zachary~A. Kent, \emph{Special
  {$L$}-values and periods of weakly holomorphic modular forms}, Proc. Amer.
  Math. Soc. \textbf{142} (2014), no.~10, 3425--3439.

\bibitem[BFOR17]{book}
Kathrin Bringmann, Amanda Folsom, Ken Ono, and Larry Rolen, \emph{Harmonic
  {M}aass forms and mock modular forms: theory and applications}, American
  Mathematical Society Colloquium Publications, vol.~64, American Mathematical
  Society, Providence, RI, 2017.

\bibitem[BO06]{BO2}
Kathrin Bringmann and Ken Ono, \emph{The {$f(q)$} mock theta function
  conjecture and partition ranks}, Invent. Math. \textbf{165} (2006), no.~2,
  243--266.

\bibitem[BO10]{BO1}
\bysame, \emph{Dyson's ranks and {M}aass forms}, Ann. of Math. (2) \textbf{171}
  (2010), no.~1, 419--449.

\bibitem[Boo15]{Boo}
Andrew~R. Booker, \emph{{$L$}-functions as distributions}, Math. Ann.
  \textbf{363} (2015), no.~1-2, 423--454.

\bibitem[Bro18]{Br}
Francis Brown, \emph{A class of non-holomorphic modular forms {I}}, Res. Math.
  Sci. \textbf{5} (2018), no.~1, Paper No. 7, 40.

\bibitem[Bum97]{B}
Daniel Bump, \emph{Automorphic forms and representations}, Cambridge Studies in
  Advanced Mathematics, vol.~55, Cambridge University Press, Cambridge, 1997.

\bibitem[CS17]{CS}
Henri Cohen and Fredrik Str\"{o}mberg, \emph{Modular forms}, Graduate Studies
  in Mathematics, vol. 179, American Mathematical Society, Providence, RI,
  2017, A classical approach. \MR{3675870}

\bibitem[DD20]{DD}
Nikolaos Diamantis and Joshua Drewitt, \emph{Period functions associated to
  real-analytic modular forms}, Res. Math. Sci. \textbf{7} (2020), no.~3, Paper
  No. 21, 23.

\bibitem[DR22]{DR}
Nikolaos Diamantis and L.~Rolen, \emph{L-values of harmonic maass forms},
  arXiv:2201.10193v3 (2022), 1--23.

\bibitem[DSKS21]{ShS}
K.~Deo~Shankhadhar and R.~Kumar~Singh, \emph{An analogue of {W}eil's {C}onverse
  {T}heorem for harmonic {M}aass forms of polynomial growth}, arXiv:2101.03101.
  (2021), 1--28.

\bibitem[EMOT54]{ERD}
A.~Erd\'{e}lyi, W.~Magnus, F.~Oberhettinger, and F.~G. Tricomi, \emph{Tables of
  integral transforms. {V}ol. {I}}, McGraw-Hill Book Company, Inc., New
  York-Toronto-London, 1954, Based, in part, on notes left by Harry Bateman.

\bibitem[GS08]{MR2429900}
Jes\'{u}s Guillera and Jonathan Sondow, \emph{Double integrals and infinite
  products for some classical constants via analytic continuations of {L}erch's
  transcendent}, Ramanujan J. \textbf{16} (2008), no.~3, 247--270.

\bibitem[Maa83]{Maa}
H.~Maass, \emph{Lectures on modular functions of one complex variable}, second
  ed., Tata Institute of Fundamental Research Lectures on Mathematics and
  Physics, vol.~29, Tata Institute of Fundamental Research, Bombay, 1983, With
  notes by Sunder Lal.

\bibitem[MS04]{MS}
Stephen~D. Miller and Wilfried Schmid, \emph{Summation formulas, from {P}oisson
  and {V}oronoi to the present}, Noncommutative harmonic analysis, Progr.
  Math., vol. 220, Birkh\"{a}user Boston, Boston, MA, 2004, pp.~419--440.

\bibitem[MSSU20]{MSSU}
Tadashi Miyazaki, Fumihiro Sato, Kazunari Sugiyama, and Takahiko Ueno,
  \emph{Converse theorems for automorphic distributions and {M}aass forms of
  level {$N$}}, Res. Number Theory \textbf{6} (2020), no.~1, Paper No. 6, 59.

\bibitem[NO20]{NO}
Michael Neururer and Thomas Oliver, \emph{Weil's converse theorem for {M}aass
  forms and cancellation of zeros}, Acta Arith. \textbf{196} (2020), no.~4,
  387--422.

\bibitem[OLBC10]{NIST}
Frank W.~J. Olver, Daniel~W. Lozier, Ronald~F. Boisvert, and Charles~W. Clark
  (eds.), \emph{N{IST} handbook of mathematical functions}, U.S. Department of
  Commerce, National Institute of Standards and Technology, Washington, DC;
  Cambridge University Press, Cambridge, 2010, With 1 CD-ROM (Windows,
  Macintosh and UNIX).

\bibitem[Raz77]{R}
Michael~J. Razar, \emph{Modular forms for {$G_{0}(N)$} and {D}irichlet series},
  Trans. Amer. Math. Soc. \textbf{231} (1977), no.~2, 489--495.

\bibitem[Shi73]{Sh}
Goro Shimura, \emph{On modular forms of half integral weight}, Ann. of Math.
  (2) \textbf{97} (1973), 440--481.

\end{thebibliography}
} 

\end{document}